\documentclass[12pt]{article}
\usepackage{amsfonts}
\usepackage[dvips]{graphics}
\usepackage[centertags]{amsmath}
\usepackage{amsfonts}
\usepackage{amssymb}
\usepackage{amsthm}
\usepackage{newlfont}
\usepackage{mathrsfs}
\usepackage[dvips]{graphics}
\usepackage[centertags]{amsmath}
\usepackage{amsfonts}
\usepackage{amssymb}
\usepackage{amsthm}
\usepackage{newlfont}
\usepackage{hyperref}
\usepackage{cases}

\newlength{\defbaselineskip}
\theoremstyle{plain}
\newtheorem{thm}{Theorem}[section]
\newtheorem{cor}[thm]{Corollary}
\newtheorem{pro}[thm]{Problem }
\newtheorem{lem}[thm]{Lemma}

\theoremstyle{definition}
\newtheorem{defn}{Definition}[section]
\newtheorem{ass}{Assumption}[section]
\newtheorem{rmk}{Remark}[section]

\renewcommand{\d}{\mathrm{d}}
\newcommand{\eps}{\varepsilon}

\makeatletter\@addtoreset{equation}{section} \makeatother
\begin{document}

\title{  Partially Observed Optimal  Control for
 Mean-Field SDEs
\thanks{This work was supported by the Natural Science Foundation of Zhejiang Province
for Distinguished Young Scholar  (No.LR15A010001),  and the National Natural
Science Foundation of China (No.11471079, 11301177)}}
\date{}
\author{\textbf{Maoning Tang}, \textbf{Qingxin Meng
}\thanks{{\small Corresponding author. E-mail address: mqx@zjhu.edu.cn,
}}\\
{\small {Department of Mathematics, Huzhou University, Zhejiang 313000, China}}} \maketitle

\begin{abstract}

In this paper, we are  concerned  with
a stochastic  optimal control  problem of
mean-field type  under  partial observation, where the state equation is
governed by the  controlled
nonlinear mean-field stochastic
  differential equation, moreover the observation noise is allowed to enter into
the  state equation   and the observation
coefficients may depend  not only on the
control process and  but also on  its probability distribution.
Under standard assumptions on the coefficients, by dual analysis and
  convex  variation, we establish
the maximum principle for optimal
control in a strong sense as well as a weak one, respectively.
 As an application,
a partially observed   linear quadratic
control  problem of mean-field type  is
studied detailed and the corresponding
dual characterization  and state feedback
presentation of the partially observed  optimal control are  obtained by the stochastic maximum principles and the classic  technique of completing squares.

\end{abstract}


\section{ Introduction}

\subsection{Basic Notations}
In this subsection, we introduce
some basic notations  which will be
used in this paper.
Let ${\cal T} : = [0, T]$ denote a finite time index, where $0<T <
\infty$. We consider a complete probability space $( \Omega,
{\mathscr F}, {\mathbb P} )$  equipped with two one-dimensional
standard Brownian motions $\{W(t), t \in {\cal T}\}$ and $\{Y(t),t \in {\cal T}\},$
respectively. Let $%
\{\mathscr{F}^W_t\}_{t\in {\cal T}}$ and $%
\{\mathscr{F}^Y_t\}_{t\in {\cal T}}$ be $\mathbb P$-completed natural
filtration generated by $\{W(t), t\in {\cal T}\}$ and $\{Y(t), t\in {\cal T}\},$ respectively. Set $\{\mathscr{F}_t\}_{t\in {\cal T}}:=\{\mathscr{F}^W_t\}_{t\in {\cal T}}\bigvee
\{\mathscr{F}^Y_t\}_{t\in {\cal T}}, \mathscr F=\mathscr F_T.$   Denote by $\mathbb E[\cdot]$ the expectation
under the probablity $\mathbb P.$
 Let $E$ be a Euclidean space. The inner product in $E$ is denoted by
$\langle\cdot, \cdot\rangle,$ and the norm in $ E$ is denoted by $|\cdot|.$
Let $A^{\top }$ denote the
transpose of the matrix or vector $A.$
For a
function $\phi:\mathbb R^n\longrightarrow \mathbb R,$ denote by
$\phi_x$ its gradient. If $\phi: \mathbb R^n\longrightarrow \mathbb R^k$ (with
$k\geq 2),$ then $\phi_x=(\frac{\partial \phi_i}{\partial x_j})$ is
the corresponding $k\times n$-Jacobian matrix. By $\mathscr{P}$ we
denote the
predictable $\sigma$ field on $\Omega\times [0, T]$ and by $\mathscr %
B(\Lambda)$ the Borel $\sigma$-algebra of any topological space
$\Lambda.$ In the follows, $K$ represents a generic constant, which
can be different from line to line.

 Next we introduce some spaces of random variable and stochastic
 processes. For any $\alpha, \beta\in [1,\infty),$ we let

$\bullet$~~$M_{\mathscr{F}}^\beta(0,T;E):$ the space of all $E$-valued and ${%
\mathscr{F}}_t$-adapted processes $f=\{f(t,\omega),\ (t,\omega)\in \cal T
\times\Omega\}$ satisfying
$
\|f\|_{M_{\mathscr{F}}^\beta(0,T;E)}\triangleq{\left (\mathbb E\bigg[\displaystyle%
\int_0^T|f(t)|^ \beta dt\bigg]\right)^{\frac{1}{\beta}}}<\infty. $

$\bullet$~~$S_{\mathscr{F}}^\beta (0,T;E):$ the space of all $E$-valued and ${%
\mathscr{F}}_t$-adapted c\`{a}dl\`{a}g processes $f=\{f(t,\omega),\
(t,\omega)\in {\cal T}\times\Omega\}$ satisfying $
\|f\|_{S_{\mathscr{F}}^\beta(0,T;E)}\triangleq{\left (\mathbb E\bigg[\displaystyle\sup_{t\in {\cal T}}|f(t)|^\beta \bigg]\right)^{\frac
{1}{\beta}}}<+\infty. $

$\bullet$~~$L^\beta (\Omega,{\mathscr{F}},P;E):$ the space of all
$E$-valued random variables $\xi$ on $(\Omega,{\mathscr{F}},P)$
satisfying $ \|\xi\|_{L^\beta(\Omega,{\mathscr{F}},P;E)}\triangleq
\sqrt{\mathbb E|\xi|^\beta}<\infty. $

$\bullet$~~$M_{\mathscr{F}}^\beta(0,T;L^\alpha (0,T; E)):$ the space of all $L^\alpha (0,T; E)$-valued and ${%
\mathscr{F}}_t$-adapted processes $f=\{f(t,\omega),\ (t,\omega)\in[0,T]%
\times\Omega\}$ satisfying $
\|f\|_{\alpha,\beta}\triangleq{\left\{\mathbb E\bigg[\left(\displaystyle
\int_0^T|f(t)|^\alpha
dt\right)^{\frac{\beta}{\alpha}}\bigg]\right\}^{\frac{1}{\beta}}}<\infty. $

\subsection{Formulation of Optimal control Problem of Mean-Field Type Under  Partial Observation}

In this subsection, under  partial observations,  we formulate two class of   optimal control problems of mean-field type in a weak form and a strong form, respectively. On probability space $( \Omega,
{\mathscr F}, {\mathbb P} ),$  we consider
the following controlled
mean-field stochastic differential equation
\begin{equation}
\displaystyle\left\{
\begin{array}{lll}
dx(t) & = & b(t,x(t),\mathbb E [ x(t)],u(t), \mathbb E [ u(t)])dt+\displaystyle
g(t,x(t),\mathbb E [ x(t)],u(t),
\mathbb E[u(t)])dW(t)
\\&&+\displaystyle {\tilde g}(t,x(t),\mathbb E [ x(t)], u(t), \mathbb E [ u(t)])dW^u(t), \\
\displaystyle x(0) & = & a\in \mathbb{R}^n,
\end{array}%
\right.  \label{eq:1.1}
\end{equation}%
with an obvervation
\begin{equation}\label{eq:1.2}
\displaystyle\left\{
\begin{array}{lll}
dY(t) & = & h(t,x(t),\mathbb E [ x(t)], u(t), \mathbb E [ u(t)])dt+dW^u(t),  \\
\displaystyle y(0) & = & 0
,\end{array}
\right.
\end{equation}
where $b: {\cal T} \times \Omega \times {\mathbb R}^n \times
{\mathbb R}^n
 \times U \times U \rightarrow {\mathbb R}^n$, $g:
{\cal T} \times \Omega \times {\mathbb R}^n \times {\mathbb R}^n
 \times U \times U\rightarrow {\mathbb R}^n$, $\tilde g: {\cal T} \times \Omega \times {\mathbb R}^n \times {\mathbb R}^n
 \times U \times U\rightarrow {\mathbb R}^n $, $h: {\cal T} \times \Omega \times {\mathbb R}^n \times {\mathbb R}^n
 \times U \times U \rightarrow {\mathbb R}$,  are given random mapping with
 $U$ being a nonempty convex subset of $\mathbb R^k.$ In the above equations, $u(\cdot)$ is our admissible control
 process defined as follows.

 \begin{defn}
   An admissible control process is defined as a
   stochastic  process $u:\cal T\times  \Omega\longrightarrow  U$ which is
   $\{\mathscr{F}^Y_t\}_{t\in \cal T}$
-adapted   and satisfies
\begin{equation} \label{eq:1.33}
  \mathbb E\bigg[\bigg(\int_0^T|u(t)|^2
dt\bigg)^{2}\bigg]<\infty.
\end{equation}
The set of all admissible controls
is denoted by $U_{ad}^W.$
\end{defn}
\begin{rmk}
In the literature (see, e.g., Tang(1998)), we know that a control process is said to be partially observed
if the control is nonanticipative  functional of the observation $Y(\cdot).$ A set of controls is said to be partially observed if
its element is partially observed. Obviously, the set $U_{ad}^W$ of all admissible control is partially observed.
\end{rmk}
 Now we make the following standard  assumptions
 on the coefficients of the equations \eqref{eq:1.1}
 and \eqref{eq:1.2}.

 \begin{ass}\label{ass:1.1}
The coefficients $b$, $g,\tilde g$ and $h$  are ${\mathscr P} \otimes {\cal
B} ({\mathbb R}^n) \otimes {\mathscr B} ({\mathbb R}^n) \otimes {\mathscr B}
(U) \otimes {\mathscr B}
(U) $-measurable. For each $(x,y,
u,v) \in \mathbb {R}^n\times \mathbb R^n \times U \times U$, $b (\cdot, x,y,
u,v), g(\cdot,x,y,u, v)$, $\tilde g (\cdot, x, y, u,v)$
and $h(\cdot, x, y, u,v)$ are all $\{\mathscr{F}_t\}_{t\in \cal T}$-adapted processes. For almost all $(t, \omega)\in \cal T
\times \Omega$, the mapping
\begin{eqnarray*}
(x,y,u,v) \rightarrow \varphi(t,\omega,x, y, u,v)
\end{eqnarray*}
 is continuous differentiable with respect to $(x, y,u,v)$ with
appropriate growths, where $\varphi=b, g,  \tilde g$ and $h.$  More precisely, there exists a constant $C
> 0$  such that for  all $x,y\in \mathbb
R^n, u,v\in U$ and a.e. $(t, \omega)\in {\cal T} \times \Omega,$
\begin{eqnarray*}
\left\{
\begin{aligned}
& (1+|x|+|y|+|u|+|v|)^{-1}|\phi(t,x,y,u,v)|
+|\phi_x(t,x,y,u,v)|
\\&\quad\quad+|\phi_y(t,x,y,u,v)|
+|\phi_u(t,x,y,u,v)|
+|\phi_v(t,x,y,u,v)|\leq C, \varphi=b, g,\tilde g;
\\&|h(t,x,y,u,v)| +|h_{x}(t,x,y,u,v)|+|h_{y}(t,x,y,u,v)|+|h_{u}(t,x,y,u,v)|+|h_{v}(t,x,y,u,v)|\leq
C.
\end{aligned}
\right.
\end{eqnarray*}
\end{ass}
Now under Assumption \ref{ass:1.1}, we
begin to discuss  the well- posedness
of \eqref{eq:1.1} and \eqref{eq:1.2}.
Indeed, putting \eqref{eq:1.2} into the state equation \eqref{eq:1.1}, we get that
\begin{equation} \label{eq:1.3}
\displaystyle\left\{
\begin{array}{lll}
dx(t) & = & [(b-{\tilde g}h)(t,x(t),\mathbb E [ x(t)],u(t), \mathbb E [ u(t)])]dt
+\displaystyle
g(t,x(t),\mathbb E [ x(t)],u(t), \mathbb E [ u(t)])dW(t)\\&&
 +\displaystyle {\tilde g}(t,x(t),\mathbb E [ x(t)],u(t), \mathbb E [ u(t)])dY(t),
\\
\displaystyle x(0) & = & a.
\end{array}
\right.
\end{equation}
Under Assumption \ref{ass:1.1}, for any $u(\cdot)\in U_{ad}^W,$  by
Lemma \ref{lem:3.3} below,
\eqref{eq:1.3} admits a strong solution
$x(\cdot)\equiv x^u(\cdot)\in S^4_{\mathscr F}(0, T; \mathbb R^n).$  On the other hand, for any $u(\cdot)\in U_{ad}^W$ associated with
the corresponding solution $x^u(\cdot)$ of
\eqref{eq:1.3},  introduce  a stochastic
process $Z^u(\cdot)$  defined by the
unique solution  of the following mean-field
SDE \begin{equation}\label{eq:3.4}
\displaystyle\left\{
\begin{array}{lll}
dZ^u(t) & = & Z^u(t)h(t,x^u(t), \mathbb E [ x(t)], u(t), \mathbb E [ u(t)])dY(t), \\
\displaystyle Z^u(0) & = & 1.
\end{array}%
\right.
\end{equation}
Define a new probability measure $\mathbb P^u$ on $(\Omega, \mathscr F)$ by
$d\mathbb P^u=Z^u(1)d\mathbb P.$  Then from Girsanov's theorem and \eqref{eq:1.2}, $(W(\cdot),W^u(\cdot))$ is an
$\mathbb R^2$-valued standard Brownian motion defined in the new probability
space $(\Omega, \mathscr F, \{\mathscr{F}_t\}_{0\leq t\leq T},\mathbb P^u).$
So $(\mathbb P^u, X^u(\cdot), Y(\cdot), W(\cdot), W^u(\cdot))$ is a weak
solution on $(\Omega, \mathscr F, \{\mathscr{F}_t\}_{t\in \cal
T})$ of  \eqref{eq:1.1} and
\eqref{eq:1.2}.

 Now for  any given admissible control
 $ u(\cdot)\in U_{ad}^W $ and the  corresponding
 weak solution  $(\mathbb P^u, x^u(\cdot), Y(\cdot), W(\cdot), W^u(\cdot))$ of \eqref{eq:1.1} and
 \eqref{eq:1.2}, we introduce the following  cost functional in the weak form,
\begin{equation}\label{eq:1.6}
\begin{split}
J(u(\cdot ))=& \mathbb E^u\displaystyle\bigg[%
\int_{0}^{T}l(t,x(t),\mathbb E [ x(t)], u(t), \mathbb E [ u(t)])dt \\
& ~~~~~+m(X(T),\mathbb E [ x(T)])\bigg],
\end{split}%
\end{equation}
where $\mathbb E^u$ denotes the expectation with respect to the
probability space $(\Omega, \mathscr F, \{\mathscr{F}_t\}_{0\leq
t\leq T},\mathbb P^u)$ and  $l:
{\cal T} \times \Omega \times {\mathbb R}^n \times {\mathbb R}^n
 \times U \times U\rightarrow {\mathbb R},$ $m: \Omega \times {\mathbb R}^n \times {\mathbb R}^n \rightarrow {\mathbb R}$  are given random mappings
satisfying  the following assumption:

\begin{ass}\label{ass:1.2}
 $l$ is ${\mathscr P} \otimes {\cal
B} ({\mathbb R}^n) \otimes {\mathscr B} ({\mathbb R}^n) \otimes {\mathscr B}
(U) \otimes {\mathscr B}
(U) $-measurable, and $m$ is ${\cal F}_T \otimes {\mathscr B} ({\mathbb
R}^n) \otimes {\mathscr B} ({\mathbb R}^n)$-measurable. For each $(x,y,
u,v) \in \mathbb {R}^n\times \mathbb R^n \times U \times U$, $f (\cdot, x, y, u,v)$ is  an ${\mathbb F}$-adapted process, and $m(x,y)$
is an ${\cal F}_{T}$-measurable random variable. For almost all $(t, \omega)\in [0,T]
\times \Omega$, the mappings
\begin{eqnarray*}
(x,y,u,v) \rightarrow l(t,\omega,x, y, u,v)
\end{eqnarray*}
and
\begin{eqnarray*}
(x,y) \rightarrow m(\omega,x, y)
\end{eqnarray*}
are continuous differentiable with respect to $(x, y,u,v)$ with
appropriate growths, respectively.
More precisely, there exists a constant $C
> 0$  such that for  all $x,y\in \mathbb
R^n, u,v\in U$ and a.e. $(t, \omega)\in [0,T]
\times \Omega,$
\begin{eqnarray*}
\left\{
\begin{aligned}
&
(1+|x|+|y|+|u|+|v|)^{-1}\left (|l_x(t,x,y,u,v)|+|l_y(t,x,y,u,v)|+
|l_u(t,x,y,u,v)|+|l_v(t,x,y,u,v)|\right)
\\&\quad\quad+(1+|x|^2+|y|^2+|u|^2+|v|^2)^{-1}|l(t,x,y,u,v)|
 \leq C;
\\
& (1+|x|^2+|y|^2)^{-1}|m(x,y)| +(1+|x|+|y|)^{-1}(|m_x(x,y)|+|m_y(x,y)|)\leq
C.
\end{aligned}
\right.
\end{eqnarray*}
\end{ass}
 Under  Assumption \ref{ass:1.1} and
 \ref{ass:1.2},
 by the estimates \eqref{eq:1.8} and
 \eqref{eq:1.9},  we get that

\begin{eqnarray}
  \begin{split}
    |J(u(\cdot))|\leq  & K\mathbb E\bigg[\int_0^T
    |Z^u(t)|(1+|x^u(t)|^2+|\mathbb E[x^u(t)]|^2+|u(t)|^2+|\mathbb E[u(t)]|^2\bigg]
    \\ \leq & K  \bigg\{\mathbb E\bigg[\sup_{t\in {\cal T}}
    |Z^u(t)|^2\bigg]\bigg\}^{\frac{1}{2}}
    \bigg\{ \mathbb E\bigg[\sup_{
    {t\in \cal T}}
    |x(t)|^4\bigg]+\mathbb E\bigg[\bigg(\int_0^T|u(t)|^2
dt\bigg)^{2}\bigg] +1\bigg\}^{\frac{1}{2}}
    \\  <& \infty,
  \end{split}
\end{eqnarray}
which implies that
the cost functional is well-defined.

 Then we  can put forward the following partially observed optimal control problem  in its weak formulation,
 i.e., with changing
 the reference probability space $(\Omega, \mathscr F, \{\mathscr{F}_t\}_{0\leq
t\leq T},\mathbb P^u),$ as follows.

\begin{pro}
\label{pro:1.1} Find an admissible control $\bar{u}(\cdot)\in U_{ad}^W$ such that
\begin{equation*}  \label{eq:b7}
J(\bar{u}(\cdot))=\displaystyle\inf_{u(\cdot)\in U_{ad}^W}J(u(\cdot)),
\end{equation*}
subject to the
 state equation \eqref{eq:1.1}, the
 observation equation \eqref{eq:1.2}
 and the cost functional \eqref{eq:1.6}.

\end{pro}
Obviously, according to  Bayes' formula,
 the cost functional \eqref{eq:1.6} can be rewritten as
\begin{equation}\label{eq:1.7}
\begin{split}
J(u(\cdot ))=& \mathbb E\displaystyle\bigg[%
\int_{0}^{T}Z^u(t)l(t,x(t),\mathbb E [ x(t)], u(t), \mathbb E [ u(t)])dt \\
& ~~~~~+Z^u(T)m(x(T),\mathbb E [ x(T)])\bigg].
\end{split}
\end{equation}
  Therefore, we  can translate   Problem \ref{pro:1.1}
  into  the following  equivalent optimal control problem in
  its strong formulation, i.e., without changing the  reference
probability space $(\Omega, \mathscr F, \{\mathscr{F}_t\}_{0\leq t\leq T},\mathbb P),$
where $Z^u(\cdot)$ will be regarded as
an additional  state process besides the
state process $x^u(\cdot).$

\begin{pro}
\label{pro:1.2} Find an admissible control $\bar{u}(\cdot)$ such that
\begin{equation*}  \label{eq:b7}
J(\bar{u}(\cdot))=\displaystyle\inf_{u(\cdot)\in U_{ad}^W}J(u(\cdot)),
\end{equation*}
 subject to
 the cost functional \eqref{eq:1.7}
 and  the following
 state equation
\begin{equation}
\displaystyle\left\{
\begin{array}{lll}
dx(t) & = & [(b-{\tilde
g}h)(t,x(t),
\mathbb E[x(t)],u(t), \mathbb E [ u(t)])]dt
+\displaystyle
g(t,x(t),\mathbb E [x(t)],u(t), \mathbb E [ u(t)])dW(t)
 \\&&+\displaystyle {\tilde g}(t,x(t),\mathbb E[ x(t)],u(t), \mathbb E [ u(t)])dY(t),
 \\
dZ(t) & = & Z(t)h(t,x(t),\mathbb E [ x(t)],u(t), \mathbb E [ u(t)])dY(t), \\
\displaystyle Z(0) & = & 1,%
\\
\displaystyle x(0) & = & a\in \mathbb{R}^n.
\end{array}%
\right.  \label{eq:3.7}
\end{equation}
\end{pro}
Any $\bar{u}(\cdot)\in U_{ad}^W$ satisfying above is called an optimal
control process of Problem \ref{pro:1.2} and the corresponding state
process $(\bar{x}(\cdot),\bar{Z}(\cdot))$ is called the optimal
state process. Correspondingly $(\bar{u}(\cdot);\bar{x}(\cdot),  \bar {Z}(\cdot))$ is called an optimal pair of Problem \ref{pro:1.2}.
\begin{rmk}
 The present formulation
of the partially observed optimal control problem is quite similar
to a completely observed optimal control problem; the only
difference lies in the admissible class $U_{ad}^W$ of controls.
\end{rmk}
In this paper,  provided  the  original
sate equation   \eqref{eq:1.1} and
the observation equation  \eqref{eq:1.2},
 we will also study the partially observed  optimal control
problem  in its strong formulation, i.e.
without changing the  reference
probability space $(\Omega, \mathscr F, \{\mathscr{F}_t\}_{0\leq t\leq T},\mathbb P).$
Precisely,  different from
the  cost functional \eqref{eq:1.6},
  the cost functional in this case  is defined by
\begin{equation}\label{eq:1.10}
\begin{split}
J(u(\cdot ))=& \mathbb E\displaystyle\bigg[%
\int_{0}^{T}l(t,X(t),\mathbb E [ x(t)], u(t), \mathbb E [ u(t)])dt \\
& ~~~~~+m(X(T),\mathbb E [ x(T)])\bigg].
\end{split}
\end{equation}
Note that $\mathbb E(\cdot)$ is
the expectation with  the original
probability $\mathbb P$ independent
of the control $u(\cdot).$   In
this case,   different from the
partially observed  optimal control problem in weak sense
discussed before, we do not need
require the admissible
control process satisfies \eqref{eq:1.33}. In this case, an  admissible
control process is defined as a $\{\mathscr{F}^Y_t\}_{0\leq t\leq T}$
adapted stochastic process valued in $U$
satisfying
\begin{equation} \label{eq:1.33}
  \mathbb E\bigg[\int_0^T|u(t)|^2
dt\bigg]<\infty.
\end{equation}
The set of all admissible controls in this case
is denoted by $U_{ad}^S.$

Then we  can put forward the partially observed optimal control problem   in its strong formulation as follows.

\begin{pro}
\label{pro:4.1} Find an admissible control $\bar{u}(\cdot)$ such
that
\begin{equation*}  \label{eq:b7}
J(\bar{u}(\cdot))=\displaystyle\inf_{u(\cdot)\in U_{ad}^W}J(u(\cdot)),
\end{equation*}
\end{pro}
 subject to
 the cost functional \eqref{eq:1.10}
 and  the following
 state equation
\begin{equation} \label{eq:4.2}
\displaystyle\left\{
\begin{array}{lll}
dx(t) & = & [(b-{\tilde
g}h)(t,x(t),
\mathbb E[x(t)],u(t), \mathbb E [ u(t)])]dt
+\displaystyle
g(t,x(t),\mathbb E [x(t)],u(t), \mathbb E [ u(t)])dW(t)
 \\&&+\displaystyle {\tilde g}(t,x(t),\mathbb E[ x(t)],u(t), \mathbb E [ u(t)])dY(t),
\\
\displaystyle x(0) & = & a\in \mathbb{R}^n.%
\end{array}
\right.
\end{equation}
Note that  under Assumptions\ref{ass:1.1}
and \ref{ass:1.2},
for any admissible control $u(\cdot)\in U_{ad}^S,$
by Lemma \ref{lem:3.3} below,
the state \eqref{eq:4.2} has a unique solution
$x(\cdot)\in S_{\mathscr F}^2(0,T; \mathbb R^n)$
and $J(u(\cdot)) < \infty,$  so Problem \ref{pro:4.1} is well-defined.

Before concluding this subsection, we give  the
well-posedness of the state equation as well as some useful
estimates which can be showed easily  by
the classic compression mapping
theorem combining with Gronwall's inequality and B-D-G inequality.

\begin{lem}\label{lem:3.3}
  Let Assumption \ref{ass:1.1} holds. Then for any $u(\cdot)\in
  M_{\mathscr{F}}^\beta(0,T;L^2 (0,T; \mathbb R^k)),$  the state equation \eqref{eq:3.7} admits a unique strong
  solution $( x(\cdot),  Z(\cdot))\in S_{\mathscr{F}}^\beta (0,T;\mathbb R^{n+1}).$
Moreover, we have the following
estimates:

\begin{eqnarray}\label{eq:1.8}
\begin{split}
 {\mathbb E} \bigg [ \sup_{t\in \cal T} | x(t) |^\beta \bigg
] &\leq& K\bigg \{1+|a|^\beta+\mathbb E\bigg[\Big(\int_0^T |u(t)|^2dt\Big)^{\frac{\beta}{2}}\bigg] \bigg\},
\end{split}
\end{eqnarray}
and for any $\alpha \geq 2,$
\begin{eqnarray}\label{eq:1.9}
\begin{split}
 {\mathbb E} \bigg [ \sup_{{t\in
 \cal T}} | Z(t) |^\alpha \bigg
] \leq K.
\end{split}
\end{eqnarray}
Further, if $(\bar x(\cdot), \bar Z(\cdot))$ is the unique strong solution
corresponding to another  $ \bar u(\cdot)\in M_{\mathscr{F}}^\beta(0,T;L^2 (0,T; \mathbb R^k)),$ then the following
estimate holds
\begin{eqnarray}\label{eq:1.15}
{\mathbb E} \bigg [ \sup_{t\in \cal T} | x (t) - \bar x(t) |^\beta
\bigg ] +{\mathbb E} \bigg [ \sup_{t\in \cal T} | Z (t) - \bar Z (t) |^\beta
\bigg ]  \leq  K {\mathbb E} \bigg [ \int_0^T| u(t)- \bar u (t)
|^2 dt\bigg ]^{\frac{\beta}{2}}.
\end{eqnarray}

\end{lem}

\subsection {Related Development and Contributions of
this paper }

Most recently,
 stochastic optimal control  problems
 of stochastic differential equations (SDE) of
 mean-field type have attracted  a great deal of  attention due to  its wide range of applications in economics and finance  such as mean-variance portfolio selection problems.
 As stated by Djehiche and Tembine (2016), the main feature of this class of control problem is that the cost functional,
 the coefficients of the drift and diffusion terms of the state equation depend not only on the state and the control,  but also on their probability distribution. The presence of  the mean-field term makes the control problem become to be  time-inconsistent so that the dynamic programming principle (DPP)
  does not work, which motivates to establish
   the stochastic maximum  principle (SMP) to
    solve this type of optimal control problems instead of trying extensions of DPP.
    It is well-known that adjoint equations play a critical role in the formulation of the stochastic maximum  principle.
Intuitively speaking, the adjoint equation of a controlled state equation of mean-field type is a backward
stochastic differential equation (BSDE) of mean -field type. So it is not until Buckdahn et al (2009a, 2009b)
established the results  on the mean-field BSDEs that the stochastic
maximum principle and related
 theoretical
 result and application  for the optimal control system of mean-field
type has become an important and popular  topic. we refer to interested readers to
Andersson and Djehiche (2011), Buckdahn et al (2011),
Li (2012),  Meyer-Brandis et al (2012),
Shen
and Siu (2013), Du et al(2013), Elliott (2013),
Hafayed (2013), Yong(2013),
Chala (2014), Shen et al(2014), Meng and Shen(2015)
 and  the reference therein for the various optimal control
 theory results on the mean-field models with
full observation.

A great of results on stochastic
optimal control
without mean-field  term  under   partial
observation  or partial  information
have  been obtained by  many authors
for various types of stochastic systems
via establishing the corresponding   MP and DPP. See e.g., Bensoussan (1983), Tang (1998), Baghery et al. (2007), Wu (2010), Wang and  Wu (2009),
Wang et al (2013, 2015a), and the reference therein for
more detailed discussion.

The purpose of  this paper is an
extension to the optimal control of
stochastic diffusion of
mean-field type under partial
observation (see Problem \ref{pro:1.1},
\ref{pro:4.1}).  Along this topic, due to
 the theoretical and practical
interest, recently,
it become more popular, e.g,
Wang et al (2014a, 2014b, 2015b, 2016),
 Djehiche and
 Tembine (2016), Ma and Liu (2017),
where the corresponding
maximum principles
are established and  practical
finance applications are illustrated.
Different from the above mentioned  references,
for our optimal control problem of mean-field type,
there are some distinctive  features
and contribution  worthy of
being emphasizing.
First,  our state system is a stochastic
 nonlinear system where the observation noise
$W^u(\cdot)$ is allowed to enter into
our state equation   and the observation
coefficients may depend  not only on the
control process,  but also on  its probability distribution. Therefore, our model is
more general and complicated, which leads
to that our adjoint equation is more different
and  the derivation of our main result
need more skills required.  Second,
for  Problem \ref{pro:1.1} in weak formulation, under the standard assumption on the
coefficients in which  case
the linear quadratic optimal control problem is
included,  the required integral condition for
our admissible control $u(\cdot)$  is
 \begin{equation} \label{eq:1.16}
  \mathbb E\bigg[\bigg(\int_0^T|u(t)|^2
dt\bigg)^{2}\bigg]<\infty,
\end{equation}
which is more weaker than that in
the existed reference,(cf., for example,
see,  Wang et al (2014a, 2014b, 2016 ))
where the required integral condition
for their admissible control $u(\cdot)$ is

\begin{equation} \label{eq:1.17}
  \sup_{t\in \cal T}\mathbb E\bigg[|u(t)|^8
dt\bigg]<\infty.
\end{equation}
When we require all the coefficients involved in the state equation and the cost functional
are bounded (see, for example, Djehiche and
 Tembine (2016)), the
integral condition can be weakened to the following
\begin{equation} \label{eq:1.18}
  \mathbb E\bigg[\int_0^T|u(t)|^2
dt\bigg]<\infty,
\end{equation}
but in this case,
the classic LQ problem are
not included.  Under
\eqref{eq:1.16}, our main result on
the stochastic maximum principle
 can be obtained based on
 the refined estimate \eqref{eq:1.8}-
 \eqref{eq:1.15} for
 the state equation.
 Note that for Problem \ref{pro:4.1} in
 the strong formulation, we need only
 require that  admissible control
 satisfied \eqref{eq:1.18} because
 the stochastic process $Z(\cdot)$  (see
 \eqref{eq:3.4}) is
 not involved in the cost functional (see, for example, Wang et al (2015a), Ma and Liu (2017)).
 Third, the main contribution of this  paper
 is that the corresponding
 maximum principle for the partial
 observed optimal control
 is established under our
 stochastic model of mean-field type
 by establishing a variation formula of
 the cost functional. The main idea is to get directly a variation formula in terms of the
Hamiltonian and the associated adjoint system which is a linear backward stochastic
differential equation of mean-field
and neither the
variational equation nor the corresponding Taylor type expansions of the cost functional and the state process will be  introduced.  As an application,
the LQ problem of mean-field type under
partial observation is
illustrated  and solved by the  stochastic maximum principle.  This paper can be regarded as an addition to the study of partially
 observed stochastic optimal control
 problems of mean-field type.

 The rest of this paper is organized as
 follows. In section 2,  the necessary maximum principle
 in a weak formulation is established by convex variation and adjoint calculation.  Section 3
 is devoted to deriving necessary as well as sufficient optimality
conditions for Problem \ref{pro:4.1}
in a strong formulation  in the form of stochastic maximum principles  in a unified
way.  As an application,
a partially observed   LQ problem of mean-field type   is
studied detailed and the corresponding
dual characterization  and sate feed-back
presentation of the optimal control  are obtained by
 the stochastic maximum principles established
 in section 3
  and  the classic  technique of completing squares, respectively.

\section{Stochastic Maximum Principle in Weak Formulation }

\noindent

  This section is devoted to establishing
   the stochastic maximum principle of
   Problem \ref{pro:1.1} or   Problem \ref{pro:1.2},
   i.e., establishing the necessary optimality condition of Pontryagin's type for
   an admissible control to be optimal. To this end,
for the state equation \eqref{eq:3.7},
we first introduce the
corresponding adjoint equation.
Actually, define
the Hamiltonian function $H:[0,T]\times \Omega \times
\mathbb{R}^n \times \mathbb{R}^n \times U \times U
\times \mathbb{R}^n \times \mathbb{R}^n
\times \mathbb{R}^n \times \mathbb R
\longrightarrow \mathbb{R}$ by
\begin{eqnarray} \label{eq:2.1}
\begin{split}
  &H(t,x,y,u,v,p,q, \tilde {q}, \tilde R)
\\&=\langle p,  b(t,x,y,u,v)\rangle +\langle q,  g(t,x,y,u,v)\rangle
+\langle  \tilde q, \tilde g(t,x,y,u,v)\rangle+ \tilde Rh(t,x,y, u,v)+l(t,x,y,u,v).
\end{split}
\end{eqnarray}
For the state equation \eqref{eq:3.7}
associated with any given admissible pair
$(\bar u(\cdot), \bar x(\cdot), \bar Z(\cdot)),$ the corresponding
adjoint equation is defined as follows:

\begin{numcases}{}\label{eq:3.10}
\begin{split}
d\bar r(t)&=-l(t,\bar x(t), \mathbb E [\bar x(t)], \bar u(t), \mathbb E [\bar u(t)])dt+
\bar R\left(
t\right)  dW\left(  t\right) +\bar {\tilde R}\left( t\right) dW^{\bar
u}\left( t\right),
\\
d\bar p\left(  t\right)  &=-\Big\{{H}_{x}\left(  t,\bar x(t), \mathbb E [\bar
x(t)], \bar u(t), \mathbb E [\bar u(t)]\right) + \frac{1}{\bar Z(t)}\mathbb E^{\bar u}\big[{H}_{y}\left( t,\bar x(t), \mathbb
E [\bar x(t)], \bar u(t), \mathbb E [ \bar u(t)]\right)\big]  \Big\}dt
\\&~~~~~~+\bar q
\left(  t\right)  dW\left(  t\right)  +\bar{\tilde q}\left(  t\right)
dW^{\bar u}\left( t\right),\\
\bar r(T)&=m(\bar x(T),\mathbb E[\bar x (T)]),
\\ \bar p(T)&=m_x(\bar x(T),\mathbb E[\bar x (T)])+\frac{1}{\bar Z(T)}\mathbb E^{\bar u} \left[m_x(\bar x(T),\mathbb E [\bar x(T)])\right],
\end{split}
\end{numcases}
where
\begin{equation} \label{eq:3.11}
  H(t,x,y,u,v)=:{ H}(t,x, y,u,v, \bar p(t),
  \bar q(t), \bar {\tilde q}(t),\bar{\tilde R}(t)-\tilde g(t,\bar x(t), \mathbb E
[\bar x(t)], \bar u(t), \mathbb E [\bar u(t)] )^\top \bar p(t)).
\end{equation}
Note the adjoint equation \eqref{eq:3.10}
is a mean-field backward
stochastic differential equation whose solution
consists of  an 6-tuple process $(\bar p(\cdot),\bar q(\cdot),\bar {\tilde q}(\cdot),\bar r(\cdot),\bar R(\cdot),\bar {\tilde R}(\cdot ) ).$
Under Assumptions \ref{ass:1.1} and
\ref{ass:1.2},  by Buckdahn (2009b),
it is easily to see that  the adjoint equation \eqref{eq:3.10} admits
a unique solution $(\bar p(\cdot),\bar q(\cdot),\bar {\tilde q}(\cdot),\bar r(\cdot),\bar R(\cdot),\bar {\tilde R}(\cdot) )\in S_{\mathscr{F}}^2(0,T;
\mathbb R^n)\times M_{\mathscr{F}}^2(0,T;
\mathbb R^n) \times M_{\mathscr{F}}^2(0,T;
\mathbb R^n)\times S_{\mathscr{F}}^2(0,T;
\mathbb R)\times M_{\mathscr{F}}^2(0,T;
\mathbb R) \times M_{\mathscr{F}}^2(0,T;
\mathbb R),
$  also called the adjoint process corresponding
the admissible pair $(\bar{u}(\cdot ); \bar{x}(\cdot ), \bar Z(\cdot))$.

Now we are in a position to
state our main result:  stochastic
maximum principle of  Problem \ref{pro:1.1}  or \ref{pro:1.2}.
\begin{thm} \label{thm:3.4}
Let assumptions \ref{ass:1.1} and \ref{ass:1.2} be satisfied. Let $(\bar{u}(\cdot ); \bar{x}(\cdot ), \bar Z(\cdot))$ be an optimal pair of Problem \ref{pro:1.2} associated with the adjoint
process  $(\bar p(\cdot),\bar q(\cdot),\bar {\tilde q}(\cdot),\bar r(\cdot),\bar R(\cdot),\bar {\tilde R}(\cdot) ).$
 Then
the optimality condition
\begin{eqnarray}\label{eq:3.2000}
\begin{split}
  &\Big\langle\mathbb E\big[ \bar Z(t){\bar H}_u(t)|\mathscr F_t^Y]
+\mathbb E^{\bar{u}}[{ \bar H}_v(t)], u-\bar u(t) \Big\rangle\geq 0
\end{split}
\end{eqnarray}
holds for any $u\in U$  and a.e. $(t,\omega )\in \lbrack 0,T]\times \Omega .$ Here
using the notation \eqref{eq:3.11}, we set
\begin{equation} \label{eq:3.13}
  \bar H_u(t)={ H}_u(t, x(t), \mathbb E[\bar x(t)],
  u(t),\mathbb E[u(t)]),
\end{equation}
and
\begin{equation}\label{eq:3.14}
  \bar H_v(t)={ H}_v(t, x(t), \mathbb E[\bar x(t)],
  u(t),\mathbb E[u(t)]).
\end{equation}
\end{thm}
To prove this theorem, we  first need   to
establish the variation  formula for
the cost functional \eqref{eq:1.6} or
\eqref{eq:1.7}  by   the classical  convex variation
method  and dual technique.

Since the control domain $U$ is convex, for any given admissible
control $\bar u (\cdot), u(\cdot)\in U_{ad}^W$, the following
perturbed control process $u^\epsilon (\cdot)$:
\begin{eqnarray} \label{eq:2.7}
\begin{split}
u^\epsilon (\cdot) = \bar u (\cdot) + \epsilon ( u (\cdot) - \bar u
(\cdot) ) \ , \quad  0 \leq \epsilon \leq 1 \ ,
\end{split}
\end{eqnarray}
is also an element of $U_{ad}^W$. We denote by $(x^\epsilon (\cdot),Z^\epsilon (\cdot)) $ the solution to
 the sate equation  \eqref{eq:3.7}
  corresponding to
$u^\epsilon (\cdot)$.  To unburden our notation, we will use the
following abbreviations:
\begin{eqnarray}\label{eq:3.16}
\left\{
\begin{aligned}
& m^\epsilon (T) = m( x^\epsilon (T), \mathbb E
[x^\epsilon (T)] ) \ ,   \bar m(T) =  m( \bar x (T), \mathbb E [\bar x  (T)] ) \ ,\\
& \phi^\epsilon (t) = \phi ( t, x^\epsilon (t), \mathbb E [x^\epsilon
(t)],u^\epsilon (t), \mathbb E[u^\epsilon (t)] ) \ , \quad \phi =  b, g,
\tilde g, h,l\ ,\\
& \bar \phi (t) = \phi( t, \bar x (t), \mathbb E [\bar x(t)] , \bar
u(t), \mathbb E [ u(t)] ) \ ,\quad \phi =  b, g, \tilde g, h,l,\
\\
& H^\epsilon (t) = H ( t, x^\epsilon (t), \mathbb E [x^\epsilon
(t)],u^\epsilon (t), \mathbb E[u^\epsilon (t)] ),\\
& \bar H (t) = H( t, \bar x (t), \mathbb E [\bar x(t)] , \bar
u(t), \mathbb E [ u(t)] ), 
\\
& \bar H_x (t) = H_x( t, \bar x (t), \mathbb E [\bar x(t)] , \bar
u(t), \mathbb E [ u(t)] ), \
\\
& \bar H_y (t) = H_y( t, \bar x (t), \mathbb E [\bar x(t)] , \bar
u(t), \mathbb E [ u(t)] ) \
.
\end{aligned}
\right.
\end{eqnarray}

To
establish the variation  formula for
the cost function \eqref{eq:1.6} or
\eqref{eq:1.7}, we need the following
two basic  Lemmas.
\begin{lem} \label{lem:3.2}
Let  Assumptions \ref{ass:1.1} and
 \ref{ass:1.2}  be satisfied. Then  for
any $2\leq \gamma \leq 4,$ we have
\begin{eqnarray}
&&{\mathbb E} \bigg [ \sup_{t\in \cal T} | x^\epsilon (t) - \bar
x (t) |^\gamma  \bigg ]+{\mathbb E} \bigg [ \sup_{t\in \cal T} | Z^\epsilon (t) -
\bar Z (t) |^\gamma \bigg ]= O (\epsilon^\gamma) \ .
\end{eqnarray}
\end{lem}

\begin{proof}
By the estimate  \eqref{eq:1.15}  in  Lemma \ref{lem:3.3}  and the definition
of $u^\eps(\cdot)$ ( see \eqref{eq:2.7}), we have
\begin{eqnarray}
\begin{split}
&{\mathbb E} \bigg [ \sup_{t\in \cal T} | x^\epsilon (t) - \bar
x (t) |^\gamma  \bigg ]+{\mathbb E} \bigg [ \sup_{t\in \cal T} | Z^\epsilon (t) -
\bar Z (t) |^\gamma \bigg ]
\\ &\leq K {\mathbb E} \bigg [ \int_0^T | u^\epsilon (t) - \bar u (t) |^2 d t \bigg ]^{\frac{\gamma}{2}}\\
&= K \epsilon^\gamma {\mathbb E} \bigg [ \int_0^T | u (t) - \bar u (t) |^2 d t \bigg ]^{\frac{\gamma}{2}}\\
&= O (\epsilon^\gamma) \ .
\end{split}
\end{eqnarray}
The proof is complete.
\end{proof}

Next we represent the difference $J (u^\epsilon (\cdot)) - J (\bar u
(\cdot))$ in terms of the Hamiltonian ${ H}$ and the adjoint process $(\bar p(\cdot),\bar q(\cdot),\bar {\tilde q}(\cdot),\bar r(\cdot),\bar R(\cdot),
\bar {\tilde R}(\cdot) $ as well as other relevant
expressions.
\begin{lem}\label{lem:2.3}
Let Assumptions \ref{ass:1.1} and \ref{ass:1.2} be satisfied.
Using the notations \eqref{eq:3.11} and
\eqref{eq:3.16},  we have
\begin{eqnarray}\label{eq10}
\begin{split}
&J (u^\epsilon (\cdot)) - J (\bar u(\cdot))
\\ =& {\mathbb E} ^{\bar u}
\bigg [ \int_0^T  \big ({ H}^\eps(t) - { \bar H} (t)
- \big < x^\epsilon (t) - \bar
x (t), { \bar H}_x (t) + \frac{1}{\bar Z(t)}
{\mathbb E}^{\bar{u}} [{ \bar H} _y(t)\big]
\big>  dt \bigg]
\\-&
{\mathbb E}^{\bar u}\bigg [ \int_0^T  \langle (\tilde g^\varepsilon(t)-\bar{\tilde g}(t))(h^\eps(t)-\bar
 h(t)), \bar p(t)\rangle dt\bigg]+\mathbb E \bigg[\int_0^T(Z^\eps(t)-\bar
Z(t))(l^\eps(t)-\bar l(t))dt\bigg]
\\
+&\mathbb E\bigg[\int_0^T
 \bar {\tilde R}(t)(Z^\eps(t)-Z^{\bar u}(t))(h^\eps(t)-\bar
 h(t))dt\bigg]+\mathbb E \bigg[  ( Z^\epsilon (T) - \bar Z(T)) (m^\epsilon (T) -
\bar m (T))\bigg]
\\+&{\mathbb E}^u \bigg [ m^\epsilon (T) - \bar m (T) - \left <
x^\epsilon (T) - \bar x (T), \bar m_x (T)
 + \frac{1}{\bar Z(T)}{\mathbb E}^{\bar u}  [\bar
m_{ x} (T)] \right
> \bigg ],
\end{split}
\end{eqnarray}
for any $u (\cdot), u^\eps (\cdot) \in {\cal A}$ and $\epsilon \in [0, 1]$.
\end{lem}

\begin{proof}
From the definitions of the Hamiltonian ${ H}$
(see \eqref{eq:2.1})
and the cost functional  $J (u (\cdot))$ (see
\eqref{eq:1.6} or
\eqref{eq:1.7}), it is easy to check  that
\begin{eqnarray}\label{eq:3.19}
\begin{split}
&J (u^\epsilon (\cdot)) - J (\bar u (\cdot)) \\ =&\mathbb E\displaystyle\bigg[%
\int_{0}^{T}Z^\eps(t)l^\eps(t)dt +Z^\varepsilon (T)m^\eps(T)\bigg]-\mathbb E\displaystyle\bigg[%
\int_{0}^{T}\bar Z(t)\bar l(t)dt +\bar Z (T)\bar m(T)\bigg]\\
=&\mathbb E^{\bar u}\bigg[\int_0^T(l^\eps(t)-\bar l(t))dt\bigg]+\mathbb E^{\bar
u}[m^\eps(T)-\bar m(T)]
 \\&+\mathbb E\bigg[\int_0^Tl^\eps(t)(\bar
Z^\eps(t)-\bar Z(t))dt\bigg]+\mathbb E \big[m^\eps(T)(\bar Z^\eps(T)-\bar Z(T))\big].
\\=&
 {\mathbb E}^{\bar u}\bigg[
 \int_0^T \Big( { H}^\eps(t) - { \bar H} (t)
 - \langle b^\epsilon (t) - \bar b (t), p (t)\rangle - \langle g^\epsilon (t) - \bar g (t), q (t)\rangle
- \langle \tilde g^\epsilon (t) - {\bar {\tilde g}} (t), \tilde
q(t)\rangle
\\&-\langle h^\eps(t)-\bar h(t),\tilde{R}(t)-{\bar {\tilde
g}^\top(t)}p(t)\rangle \Big) dt\bigg]+\mathbb E^{\bar u}[m^\eps(T)-\bar
m(T)]\\& +\mathbb E\bigg[\int_0^Tl^\eps(t)(\bar Z^\eps(t)-\bar
Z(t))dt\bigg]+\mathbb E[m^\eps(T)(\bar Z^\eps(T)-\bar Z(T))] .
\end{split}
\end{eqnarray}
From \eqref{eq:1.1} and the relation \eqref{eq:1.2}, it is easily to see that  $x^\eps(\cdot)-x(\cdot)$ satisfies
the following mean-field SDE
\begin{equation}
\displaystyle\left\{
\begin{array}{ll}
&d(x^\eps(t)-x(t)) =  [b^\eps(t)-\bar b(t)]dt+\displaystyle
[g^\eps(t)-\bar g(t)]dW(t)+ [\tilde g^\eps(t)-\bar {\tilde g}(t)]dW^{\bar u}(t)
\\&\quad \quad \quad\quad \quad
\quad \quad- \tilde g^\eps(t)[h^\eps(t)-\bar h(t)]dt
, \\
&\displaystyle x^\eps(0)-\bar x(0)=  0.
\end{array}%
\right.  \label{eq:4.10}
\end{equation}
From \eqref{eq:3.10}, we know that
$(\bar p(\cdot),\bar q(\cdot),\bar {\tilde q}(\cdot) )$ satisfies the
following  mean-field BSDE

\begin{numcases}{}\label{eq:3.3}
\begin{split}
d\bar p\left(  t\right)  &=-\Big\{{\bar H}_{x}\left( t \right) +\frac{1}{\bar Z(t)}\mathbb E^{\bar u}\big[{\bar H}_{y}(t)\big]  \Big\}dt+\bar q
\left(  t\right)  dW\left(  t\right)  +\bar{\tilde q}\left(  t\right)
dW^{\bar u}\left( t\right),
\\\bar p(T)&=\bar m_x(T)+\frac{1}{\bar Z(T)}\mathbb E^{\bar u} [\bar m_y(T)].
\end{split}
\end{numcases}
Applying It\^o's formula to $\langle  x^\epsilon (t) - \bar x (t), \bar p (t) \rangle $ and takeing
expectation under the probability $\mathbb P^{\bar u}$
results in
\begin{eqnarray}\label{eq:3.22}
  &&{\mathbb
E}^{\bar u} \bigg [ \int^T_0  \langle b^\epsilon (t) - \bar b (t), p
(t)\rangle +\langle g^\epsilon (t) - \bar g (t), q (t)\rangle +
\langle \tilde g^\epsilon (t) - {\bar {\tilde g}} (t), \tilde
q(t)\rangle \bigg]dt \nonumber
\\=&&
\mathbb E^{\bar u}\bigg[\int_0^T\big < x^\epsilon (t) - \bar x (t), { H}_x (t) +\frac{1}{\bar Z(t)} {\mathbb E}^{\bar u} [{ H} _y(t)]  \big>dt\bigg]
\\&&+{\mathbb E}^{\bar u} \big [  \left <
x^\epsilon (T) - \bar x (T), \bar m_x (T)
 + \frac{1}{\bar Z(T)} {\mathbb E}^{\bar u} [ \bar m_{y}
(T)]  \right
> \big ]\nonumber
\\&&+\mathbb E^{\bar u}\bigg[\int_0^T \Big\langle (h^\eps(t)-\bar h(t)) {\tilde
g}^{{\eps}}(t), p(t)\Big\rangle dt\bigg ].
\end{eqnarray}
On the other hand, from \eqref{eq:3.4}, it is easy to check that  $Z^\eps(\cdot)-\bar Z(\cdot)$
satisfies the following mean-field SDE
\begin{equation}
\displaystyle\left\{
\begin{array}{lll}
d(Z^\eps(t)-\bar Z(t)) & = & (Z^\eps (t)h^\eps(t)-\bar Z(t)\bar h(s))dY(t), \\
\displaystyle Z^\eps(0)-\bar Z(0)  & = & 0,
\end{array}
\right.  \label{eq:4.13}
\end{equation}
and from \eqref{eq:3.10}, $(\bar r(\cdot),\bar R(\cdot),\bar {\tilde R}(\cdot) $
satisfies the following mean-field
BSDE
\begin{numcases}{}\label{eq:3.24}
\begin{split}
d\bar r(t)&=-\bar l(t)dt+\bar R\left(
t\right)  dW\left(  t\right) +\bar {\tilde R}\left( t\right) dY\left( t\right)-\bar{\tilde R}\left( t\right)
\bar h(t) dt,\\
\bar r(T)&=\bar m(T).
\end{split}
\end{numcases}
Applying It\^o's formula to $(Z^\eps (t) - \bar Z (t)) \bar r (t) $ and taking expectation
under the probability    $P $
results in
\begin{eqnarray}\label{eq:3.25}
\begin{split}
 &\mathbb E[(Z^\eps(T)-\bar Z(T))\bar m(T)]+\mathbb E\bigg[\int_0^T\bar l(t)(\bar
Z^\eps(t)-\bar Z(t))dt\bigg]
\\&=\mathbb E\bigg[\int_0^T \bar{\tilde
R}(t) Z^\eps (t)(h^\eps(t)-\bar h(s))dt\bigg].
\end{split}
\end{eqnarray}
Now  putting \eqref{eq:3.22} and \eqref{eq:3.25}  into \eqref{eq:3.19}, we deduce that
\eqref{eq10} holds. The proof is complete.
\end{proof}

Now we are in the position to use Lemma \ref{lem:3.2} and Lemma
\ref{lem:2.3} to derive the variational formula for the cost functional
$J(u(\cdot))$ in terms of the Hamiltonian ${ H}$.

\begin{thm}\label{them:3.7}
 Let Assumptions \ref{ass:1.1} and \ref{ass:1.2}
 be satisfied.  Let $(\bar u(\cdot), \bar x(\cdot),\bar Z(\cdot))$
 and $(u(\cdot), x(\cdot), Z(\cdot))$
 be two any given admissible pair.
 And let $(\bar p(\cdot),\bar q(\cdot),\bar {\tilde q}(\cdot),\bar r(\cdot),\bar R(\cdot),\bar {\tilde R}(\cdot) )$ be
  the adjoint process corresponding
 to the admissible pair $(\bar u(\cdot), \bar x(\cdot), \bar Z(\cdot)).$
 Then  for  the cost functional
 \eqref{eq:1.6} or \eqref{eq:1.7},using
 the notations \eqref{eq:3.13}, \eqref{eq:3.14} and
 \eqref{eq:3.16},
 we  have  the following
 variation formula:
 \begin{eqnarray}
&& \frac{d}{d\epsilon} J ( \bar u (\cdot) + \epsilon ( u (\cdot) - \bar u (\cdot) ) ) |_{\epsilon=0}  := \lim_{\epsilon \rightarrow 0^+}
\frac{ J ( \bar u (\cdot) + \epsilon ( u (\cdot) - \bar u (\cdot) ) ) - J( \bar u (\cdot) ) }{\epsilon} \nonumber \\
&& = {\mathbb E} \bigg [ \int_0^T \left <
\bar Z(t)\bar { H}_u (t)+\mathbb E^{\bar u}[\bar H_v(t)], u (t) - \bar u
(t) \right > d t \bigg ] \ .
\end{eqnarray}
\end{thm}

\begin{proof}

Set
\begin{eqnarray}
\begin{split}
u^\epsilon (\cdot) = \bar u (\cdot) + \epsilon ( u (\cdot) - \bar u
(\cdot) ) \ , \quad  0 \leq \epsilon \leq 1 \ .
\end{split}
\end{eqnarray}
Let  $(x^\eps (\cdot), Z^\eps(\cdot))$ be the state
process corresponding to $u^\eps(\cdot).$
For notational simplicity,  using
 the notations \eqref{eq:3.13}, \eqref{eq:3.14} and
 \eqref{eq:3.16}, we write
\begin{eqnarray}
\begin{split}
\beta^\epsilon_1 :=&{\mathbb E} ^{\bar u} \bigg[ \int_0^T \Big ( { H}^\eps(t) - {
\bar H} (t) -
\big < x^\epsilon (t) - \bar x (t), { \bar H}_x (t) + \frac{1}{\bar Z(t)}{\mathbb E}^{\bar u} [{ \bar H} _y(t)] \big>
\\&\quad \quad- \big < u^\epsilon (t) -
\bar u (t), { \bar H}_u (t)+ \frac{1}{\bar Z(t)}\mathbb E^{\bar u}[{ \bar H}_v (t)]
\big>\Big) dt\bigg],
\end{split}
\end{eqnarray}

\begin{eqnarray}
\begin{split}
\beta^\epsilon_2 :={\mathbb E}^{\bar u}\bigg [ \int_0^T  \langle (\tilde g^\varepsilon(t)-\bar{\tilde g}(t))(h^\eps(t)-\bar
 h(t)), \bar p(t)\rangle dt\bigg],
\end{split}
\end{eqnarray}

\begin{eqnarray}
\begin{split}
\beta^\epsilon_3 :={\mathbb E}^u \big [ m^\epsilon (T) - \bar m (T) - \left <
x^\epsilon (T) - \bar x (T), \bar m_x (T) +
 \frac{1}{\bar Z(T)}{\mathbb E}^{\bar u}[\bar
m_{ x} (T)]  \right
> \big ],
\end{split}
\end{eqnarray}

\begin{eqnarray}
\begin{split}
\beta^\epsilon_4 :=\mathbb E \bigg[\int_0^T(\bar Z^\eps(t)-\bar
Z(t))(l^\eps(t)-\bar l(t))dt\bigg],
\end{split}
\end{eqnarray}

\begin{eqnarray}
\begin{split}
\beta^\epsilon_5 :=\mathbb E\bigg[\int_0^T \tilde R(t)(Z^\eps(t)-\bar Z(t))(h^\eps(t)-\bar
 h(t))dt\bigg],
\end{split}
\end{eqnarray}

\begin{eqnarray}
\begin{split}
\beta^\epsilon_6 :=\mathbb E \bigg[  ( Z^\epsilon (T) - \bar Z(T)) (m^\epsilon (T) -
\bar m (T))\bigg].
\end{split}
\end{eqnarray}
By Lemma \ref{lem:2.3}, we have
\begin{eqnarray}\label{eq:3.35}
\begin{split}
&J ( u^\epsilon (\cdot) ) - J ( \bar u(\cdot) )
 \\&= \beta^\epsilon +
\epsilon {\mathbb E}^{\bar u} \bigg [ \int_0^T \left < { \bar H}_u (t)+\frac{1}{\bar Z(t)}\mathbb E^{\bar u}[{ \bar H}_v (t)] , u (t) -\bar  u (t) \right >
d t \bigg ] ,
\end{split}
\end{eqnarray}
where
\begin{eqnarray} \label{eq:3.36}
\begin{split}
\beta^\varepsilon=\beta^\varepsilon_1+\beta^\varepsilon_2+\beta^\varepsilon_3
+\beta^\varepsilon_4+\beta^\varepsilon_5+\beta^\varepsilon_6.
\end{split}
\end{eqnarray}
Now we begin to prove
\begin{equation}\label{eq:2.29}
  \beta^\eps=o(\eps)
\end{equation}
Indeed, for $\beta^\epsilon_2,$ under  Assumptions \ref{ass:1.1},
we have
\begin{eqnarray}\label{eq:3.37}
\begin{split}
|\beta^\epsilon_2| &\leq {\mathbb E}\bigg [ \int_0^T  |\tilde g^\varepsilon(t)-\bar{\tilde g}(t))||(h^\eps(t)-\bar
 h(t))|| \bar p(t)|dt\bigg]
 \\&\leq
  C{\mathbb E}\bigg [ \int_0^T  \big( | u^\eps (t) - \bar u (t) |+| \mathbb E [u^\eps (t)] -
  \mathbb E[\bar u (t)] |+
 | x^\eps (t) - \bar x (t) |+| {\mathbb E}[ x^\eps (t)] - \mathbb E[\bar x (t)]|\big)
 \\&~~~~~~~~~\quad\quad\quad \cdot|(h^\eps(t)-\bar
 h(t))||\bar  p(t)|dt\bigg]
 \\&\leq
  C\bigg\{{\mathbb E}\bigg [ \int_0^T  \big( | u^\eps (t) - \bar u (t) |^2+
 | x^\eps (t) - \bar x (t) |^2\big) dt\bigg]\bigg\}^{\frac{1}{2}}\bigg\{{\mathbb E}\bigg[\int_0^T|(h^\eps(t)-\bar
 h(t))|^2| \bar p(t)|^2dt\bigg]\bigg\}^{\frac{1}{2}}
 \\&
 \leq C\eps \bigg[\int_0^T|(h^\eps(t)-\bar
 h(t))|^2| \bar p(t)|^2dt\bigg]\bigg\}^{\frac{1}{2}}
 \\&=o(\eps),
\end{split}
\end{eqnarray}
where the last second inequality can be obtained
by  Lemma \ref{lem:3.2}
and the last inequality  can be got by  the fact that
\begin{eqnarray}
  \begin{split}
   \lim_{\eps\longrightarrow 0} \mathbb E\bigg[\int_0^T|(h^\eps(t)-\bar
 h(t))|^2| \bar p(t)|^2dt\bigg]
 =0,
  \end{split}
\end{eqnarray}
which can be obtained by the
Lemma \ref{lem:3.2} and the dominated convergence theorem,
since the function $h$ is bounded.
\\
 For  $\beta_5,$
in view of
Lemma \ref{lem:3.2} and the dominated convergence theorem. we have
\begin{eqnarray}\label{eq:3.39}
\begin{split}
|\beta^\epsilon_5|\leq& \mathbb E\bigg[\int_0^T |\bar{\tilde R}(t)||Z^\eps(t)-\bar Z(t)||h^\eps(t)-\bar
 h(t)|dt\bigg]
 \\ \leq & \mathbb E\bigg[\sup_{0\leq t \leq T}|Z^\eps(t)-\bar Z(t)| \int_0^T |
 \bar{\tilde R}(t)||h^\eps(t)-\bar
 h(t)|dt\bigg]
 \\ \leq & C \bigg\{\mathbb E\bigg[\sup_{0\leq t \leq T}|Z^\eps(t)-\bar Z(t)|^2\bigg]\bigg\}^{\frac{1}{2}} \bigg\{\mathbb E\bigg[\int_0^T |\bar{\tilde R}(t)|^2|h^\eps(t)-\bar
 h(t)|^2dt\bigg]\bigg\}^{\frac{1}{2}}
 \\ \leq & C \eps \bigg\{\mathbb E\bigg[\int_0^T |\bar{\tilde R}(t)|^2|h^\eps(t)-\bar
 h(t)|^2dt\bigg]\bigg\}^{\frac{1}{2}}
 \\  = &o(\eps).
\end{split}
\end{eqnarray}
For $\beta^\epsilon_4,$  in view of
Lemma \ref{lem:3.2} and the dominated convergence theorem, we have
\begin{eqnarray} \label{eq:3.40}
\begin{split}
|\beta^\epsilon_4 | &\leq \mathbb E \bigg[\int_0^T|\bar Z^\eps(t)-\bar
Z(t)||l^\eps(t)-\bar l(t)|dt\bigg]
\\ &\leq  \mathbb E\bigg[\sup_{0\leq t \leq T}|Z^\eps(t)-\bar Z(t)| \int_0^T |l^\eps(t)-\bar
 l(t)|dt\bigg]
 \\ &\leq \bigg\{\mathbb E\bigg[\sup_{0\leq t \leq T}|Z^\eps(t)-\bar Z(t)|^2\bigg]\bigg\}^{\frac{1}{2}} \bigg\{\mathbb E\bigg[\int_0^T |l^\eps(t)-\bar
 l(t)|dt\bigg]^2\bigg\}^{\frac{1}{2}}
 \\ &\leq C \eps \bigg\{\mathbb E\bigg[\int_0^T |l^\eps(t)-\bar
 l(t)|dt\bigg]^2\bigg\}^{\frac{1}{2}}
 \\&=o(\eps).
\end{split}
\end{eqnarray}
For $\beta^\epsilon_6, $  in view of
Lemma \ref{lem:3.2} and the dominated convergence theorem, we get
\begin{eqnarray} \label{eq:3.41}
\begin{split}
|\beta^\epsilon_6 | &\leq \mathbb E \bigg[|\bar Z^\eps(T)-\bar
Z(T)||m^\eps(T)-\bar m(T)|\bigg]
 \\ &\leq \bigg\{\mathbb E\bigg[|Z^\eps(T)-\bar Z(T)|^2\bigg]\bigg\}^{\frac{1}{2}}
 \bigg\{\mathbb E\bigg[|m^\eps(T)-m^{\bar u}(T)|^2\bigg]\bigg\}^{\frac{1}{2}}
 \\ &\leq C\eps
 \bigg\{\mathbb E\bigg[|m^\eps(T)-\bar m(T)|^2\bigg]\bigg\}^{\frac{1}{2}}
 \\&=o(\eps).
\end{split}
\end{eqnarray}
For $\beta^\epsilon_3,$  under Assumptions \ref{ass:1.1} and \ref{ass:1.2}, using the Taylor Expansions on the function $m$ with respect to
$x$ and $y$,
Lemma \ref{lem:3.2} and the dominated convergence theorem leads to

\begin{eqnarray} \label{eq:3.42}
&&|\beta^\epsilon_3 | \nonumber
\\&\leq& {\mathbb E}^{\bar u}\Big[ |
\langle
x^\epsilon (T) - \bar x (T),  m_x^{\eps, \lambda}(T)-\bar m_x (T) +\frac{1}{\bar Z(T)}\mathbb E^{\bar u}[m_y^{\eps, \lambda}(T)]- \mathbb E[\bar m_y (T)]\rangle \big |\Big] \nonumber
\\&\leq & \bigg\{
{\mathbb E}\big[|\bar Z(T)|^2\big]\bigg\}^{\frac{1}{2}}
   \bigg\{\mathbb E\bigg[|x^\eps(T)-\bar x(T)|^4\bigg]\bigg\}^{\frac{1}{4}}\bigg\{\mathbb E\bigg[|m_x^{\eps, \lambda}(T)-\bar m_x (T) +
   \mathbb E[m_y^{\eps, \lambda}(T)]
   -\mathbb E[\bar m_y (T)]|^4\bigg]\bigg\}^{\frac{1}{4}}\nonumber
   \\&\leq &  C\eps\bigg\{\mathbb E\bigg[|m_x^{\eps, \lambda}(T)-\bar m_x (T) +\mathbb E[m_y^{\eps, \lambda}(T)]-\mathbb E[\bar m_y (T)]|^4\bigg]\bigg\}^{\frac{1}{4}}\nonumber
   \\&=&o(\eps),
\end{eqnarray}
where we have used the following shorthand
notations:
$$m_x^{\eps, \lambda}(T)= \int_0^1 m_x(\bar x(T)+ \lambda (x^\eps(T)-\bar x(T)), \mathbb E[\bar x(T)]+ \lambda (\mathbb E [x^\eps(T)]-\mathbb E[\bar x(T)]))d\lambda,$$
and
$$m_y^{\eps, \lambda}(T)= \int_0^1 m_y(\bar x(T)+ \lambda (x^\eps(T)-\bar x(T)),\mathbb E[\bar x(T)]+\lambda (\mathbb E[x^\eps(T)]-\mathbb E[\bar x(T)]))d\lambda.$$
Similar to \eqref{eq:3.42},
 using the Taylor Expansions on the function $H$ with respect to
$x, y,u$ and $v$, Lemma \ref{lem:3.2} and the dominated convergence theorem, we have
\begin{eqnarray} \label{eq:3.43}
  \begin{split}
    \beta_1^\eps= o(\eps).
  \end{split}
\end{eqnarray}
Therefore,  combing
\eqref{eq:3.39}-\eqref{eq:3.43} and
using \eqref{eq:3.36},  we get that
\eqref{eq:2.29} holds.
 Then putting \eqref{eq:2.29} into
 \eqref{eq:3.35}, we have
 \begin{eqnarray}
 \begin{split}
& \frac{d}{d\epsilon} J ( \bar u (\cdot) + \epsilon ( u (\cdot) - \bar u (\cdot) ) ) |_{\epsilon=0} \nonumber \\
& = \lim_{\epsilon \rightarrow 0^+}
\frac{ J ( \bar u (\cdot) + \epsilon ( u (\cdot) - \bar u (\cdot) ) ) - J( \bar u (\cdot) ) }{\epsilon} \nonumber \\
& =  \lim_{\epsilon \rightarrow 0^+}\frac{\beta^\epsilon +
\epsilon {\mathbb E}^{\bar u} \bigg [
 \displaystyle\int_0^T \left < {\bar H}_u (t)+\mathbb E[\bar H_v(t)] , u (t) -\bar  u (t) \right >
d t \bigg ]}{\eps}
\\&= {\mathbb E} \bigg [ \displaystyle\int_0^T \left < {\bar Z(t)}{ \bar  H}_u (t)+ \mathbb E^{\bar u}[\bar H_v(t)] , u (t) -\bar  u (t) \right >
d t \bigg ]\ .
\end{split}
\end{eqnarray}
The proof is complete.
\end{proof}

Now we are ready to prove Theorem \ref{thm:3.4}

\begin{proof}
Since all admissible controls
are $\{\mathscr F^Y_t\}_{t\in \cal T}$-adapted
processes, from the property of conditional
expectation, Theorem \ref{them:3.7} and the optimality of $\bar u(\cdot)$,  we
deduce that
\begin{eqnarray*}
&& {\mathbb E} \bigg [ \int_0^T
 \langle {\mathbb E}
  [{\bar Z(t)}{ \bar  H}_u (t)
  +\mathbb E^{\bar u}[\bar H_v(t)] | {\mathscr F}_t ^Y] ,
u (t) - \bar u (t) \rangle  d t \bigg ] \\
&& = {\mathbb E} \bigg[\int_0^T \langle 
{\bar Z(t)}{ \bar  H}_u (t)+\mathbb E^{\bar u}[\bar H_v(t)]
, u (t) - \bar u (t) \rangle \d t \bigg ] \\
&& = \lim_{\epsilon \rightarrow 0} \frac{J( \bar u (\cdot) + \epsilon (
u (\cdot) - \bar u (\cdot) ) ) - J ( \bar u (\cdot) )}{\epsilon} \geq 0 ,
\end{eqnarray*}
which imply that \eqref{eq:3.2000}
holds. The proof is complete.
\end{proof}

\section{ Stochastic Maximum Principle  in Strong  Formulation }

This section is devoted to establish
the stochastic maximum principles of
Problem \ref{pro:4.1}.  In this case, the Hamiltonian
$H:[0,T]\times \Omega \times
\mathbb{R}^n \times
\mathbb{R}^n\times U \times U\times
\mathbb{R}^n\times
\mathbb{R}^n\times
\mathbb{R}^n  %
\rightarrow \mathbb{R}$ is defined  by
\begin{eqnarray} \label{eq:4.3}
\begin{split}
  H(t,x,y,u,v,p,q, \tilde {q})=&\langle p,  b(t,x,y,u,v)-\tilde g(t,x,y,u,v)h(t, x,y, u,v)\rangle
  \\&+\langle q,
g(t,x,y,u,v)\rangle +   \langle\tilde q,\tilde g(t,x,y,u,v) \rangle +l(t,x,y,u,v).
\end{split}
\end{eqnarray}
Then for any admissible pair
$(\bar u(\cdot), \bar x(\cdot)),$  the
corresponding  adjoint process
 is  defined   as   the solution to  the following mean-field
BSDE:

\begin{numcases}{}\label{eq:4.4}
\begin{split}
d\bar p\left(  t\right)  &=-\bigg[{\bar H}_{x}(t)+\mathbb E[\bar H_y(t)] \bigg ]dt+\bar q\left(  t\right)  dW\left(  t\right)  +\bar{\tilde q}\left(  t\right)
dY\left( t\right),
\\ \bar P(T)&= \bar m_x( T)+\mathbb E[ \bar m_{y}(T)],
\end{split}
\end{numcases}
where we have used the following
shorthand notation
\begin{eqnarray}\label{eq:3.3}
\left\{
\begin{aligned}
    \bar H(t)=&{ H}(t,\bar x(t), \mathbb E[\bar x(t)] ,\bar u(t),
  \mathbb E[\bar u(t)],\bar p(t),\bar q(t),
   {\bar{\tilde q}}(t)),\\
   \bar m(T)=& m(\bar x(T),  \mathbb E[\bar x(T)]).
\end{aligned}
\right.
\end{eqnarray}
Under Assumption \ref{ass:1.1} and
\ref{ass:1.2},
by Buckdahn (2009b),  \eqref{eq:4.4} admits
a unique strong  slution $(\bar p(\cdot), \bar q(\cdot),
\bar {\tilde q}(\cdot))\in S_{\mathscr{F}}^2 (0,T;\mathbb R^{n})\times M_{\mathscr{F}}^2(0,T;\mathbb R^n)
\times M_{\mathscr{F}}^2(0,T;\mathbb R^n),$
which is also called the adjoint process corresponding
to the admissible pair $(u(\cdot), x(\cdot))$

\subsection{ Sufficient Conditions of Optimality }
In this section, we are going to establish
the sufficient Pontryagin maximum principle of Problem  \ref{pro:4.1}.  To this end,
we need the following  Lemma.

\begin{lem}\label{lem:4.2}
Let Assumptions \ref{ass:1.1} and
\ref{ass:1.2} be
satisfied. Let $(u(\cdot), x(\cdot))$
and $(\bar u(\cdot), \bar x(\cdot))$
be two any given  admissible  pair  of
Problem \ref{pro:4.1}.
Let $(\bar p(\cdot), \bar q(\cdot), \bar{\tilde q}(\cdot))$
be the
adjoint process associated with
the admissible pair
$(\bar u(\cdot), \bar x(\cdot)).$
   Then  for the
     cost functional
     \eqref{eq:1.10}, using the notation
     \eqref{eq:3.3}, we have the  following presentation:
\begin{eqnarray}\label{eq:4.6}
&&J (u (\cdot)) - J (\bar u(\cdot))\nonumber
\\ &=& {\mathbb E}
\int_0^T \bigg [ { H}(t)- \bar H(t)-
\big < x^\epsilon (t) - \bar
x (t), {\bar H}_x (t)+
{\mathbb E} [{ \bar H} _y(t)]
\big>\nonumber \bigg]dt\\
&&+{\mathbb E}\big [ m (T) - \bar m (T) - \left <
x (T) - \bar x (T), \bar m_x (T) + {\mathbb E}  [\bar
m_{\overline y} (T)]  \right
> \big ],
\end{eqnarray}
where
\begin{eqnarray}\label{eq:3.50}
\left\{
\begin{aligned}
     H(t):=&{ H}(t, x(t), \mathbb E[x(t)] , u(t),
  \mathbb E[u(t)],\bar p(t),\bar q(t),
   {\bar{\tilde q}}(t)),\\
    m(T):=& m(x(T),  \mathbb E[x(T)]) .
\end{aligned}
\right.
\end{eqnarray}
\end{lem}
\begin{proof}
  Similar  to the  proof of
  Lemma \ref{lem:2.3},
  \eqref{eq:4.6} can be obtained
  by  using the definition of  the
  Hamiltonian function
  $H$ (see \eqref{eq:4.3}) and the cost functional $J(u(\cdot))$
    (see \eqref{eq:1.10}) and
    applying It\^{o} formula to
    $\langle x(t)-\bar x(t), \bar p(t) \rangle$
    and then taking expectation under  the
    probability $\mathbb P.$  Since the proof
    is standard, here we omit the concrete
    calculation.  The proof is complete.
\end{proof}

Next we give the sufficient condition of optimality for the
existence of an optimal control of Problem \ref{pro:4.1}.

\begin{thm}{\bf [Sufficient Stochastic Maximum Principle I] } \label{thm:4.3}

 Let Assumptions \ref{ass:1.1} and
\ref{ass:1.2} be
satisfied. Let $(\bar u (\cdot),
\bar x (\cdot))$  be an admissible pair associated
with the adjoint process $(\bar p(\cdot), \bar q(\cdot), \bar{\tilde q}(\cdot)).$ Suppose that

\begin{enumerate}
\item $H( t, x, y, u, v, {\bar p} (t), {\bar q} (t), \bar {\tilde q}(t) )$ is convex in $( x, y, u, v )$,
\item $m (x,y)$ is convex in $(x,y)$,
\item  For  any $u(\cdot)\in U_{ad}^S,$
\begin{eqnarray} \label{eq:4.8}
  \mathbb E\bigg[\langle  u (t) - \bar u (t), \bar H_u(t)+\mathbb E [\bar H_v(t)]\rangle\bigg]
\geq 0.
\end{eqnarray}
\end{enumerate}
Then $(\bar u (\cdot),  \bar x (\cdot))$ is an optimal pair of
Problem \ref{pro:4.1}.
\end{thm}

\begin{proof}
Let $(u (\cdot), x (\cdot))$ be an arbitrary admissible pair. In view of
Lemma \ref{lem:4.2},  we have
\begin{eqnarray}\label{eq:4.9}
&&J (u (\cdot)) - J (\bar u(\cdot))\nonumber
\\ &=& {\mathbb E}\bigg[
\int_0^T \Big ( { H}(t)- \bar H(t)-
\big < x (t) - \bar
x (t), {\bar H}_x (t)+
{\mathbb E} [{ \bar H} _y(t)]
\big>\nonumber \Big)dt\bigg]\\
&&+{\mathbb E}\big [ m (T) - \bar m (T) - \left <
x (T) - \bar x (T), \bar m_x (T) + {\mathbb E}  [\bar
m_{\overline y} (T)]  \right
> \big ].
\end{eqnarray}
 The condition 1 and 3
 lead to
\begin{eqnarray}\label{eq:3.80}
\begin{split}
\mathbb E[H(t)-\bar H(t) ]\geq &
\mathbb E\bigg[  \langle
x(t)-\bar x(t),\bar H_x(t)+\mathbb E[\bar H_y(t)]\rangle
+  \langle
u(t)-\bar u(t),\bar H_u(t)  +\mathbb E[\bar H_v(t)]
\rangle\bigg]\\
\geq &
\mathbb E\bigg[  \langle
x(t)-\bar x(t),\bar H_x(t)+\mathbb E[\bar \bar H_y(t)]\bigg].
\end{split}
\end{eqnarray}
The condition 2 arrives at
\begin{eqnarray}\label{eq:5.4}
  \mathbb E[m (T) - \bar m (T)] \geq
  \mathbb E\big[\langle x (T) -
\bar x (T), \bar m_x (T)+\mathbb E[\bar m_y (T)]  \rangle\big].
\end{eqnarray}
Putting \eqref{eq:3.80} and  \eqref{eq:5.4} into \eqref{eq:4.9}, we get
\begin{eqnarray}
J (u (\cdot)) - J (\bar u (\cdot)) \geq 0 \ .
\end{eqnarray}
 Since  $u (\cdot)$ is arbitrary,  $\bar u
(\cdot)$ is an optimal control and thus $(
\bar u (\cdot), \bar x(\cdot))$ is an optimal pair. The proof is complete.
\end{proof}

The convexity condition of  $m$ is sometimes too strong  to hold which may limit the applicability of our sufficient maximum principle. To
overcome this limitation, we note that the proof of Theorem \ref{thm:4.3} still holds as long as the terminal cost  $m$ is convex in
an expected sense. Therefore, weakening the convexity of the $m$, we provide the following  corollary of Theorem \ref{thm:4.3} as  the second sufficient maximum principle.

\begin{cor}\label{cor1}{\bf [Sufficient Stochastic Maximum principle II]}
Let  Assumption \ref{ass:1.1} and \ref{ass:1.2} be satisfied. Let $(\bar u (\cdot),
\bar x (\cdot))$  be an admissible pair associated
with the adjoint process $(\bar p(\cdot), \bar q(\cdot), \bar{\tilde q}(\cdot)).$ Suppose that
\begin{enumerate}
 \item $H( t, x, y, u, v, {\bar p} (t), {\bar q} (t), \bar {\tilde q}(t) )$ is convex in $( x, y, u, v )$,
\item  For any random variables $X_1,  X_2, \in L^2 (\Omega,{\mathscr{F}},P; \mathbb R^n),$
\begin{eqnarray*}
 {\mathbb E} \big [ m (X_1, \mathbb E [X_1]) - m (X_2, \mathbb E [X_2])\big ]  \geq {\mathbb E} \big [  \langle X_1 - X_2 ,
 m_x (X_2, \mathbb E [X_2]) + {\mathbb E} [
 m_y (X_2, \mathbb E [X_2]) \rangle  \big
]  \ ,
\end{eqnarray*}
\item  For  any $u(\cdot)\in U_{ad}^S,$
\begin{eqnarray} \label{eq:4.13}
  \mathbb E\bigg[\langle  u (t) - \bar u (t), \bar H_u(t)+\mathbb E [\bar H_v(t)]\rangle\bigg]
\geq 0,
\end{eqnarray}
\end{enumerate}
then $\bar u (\cdot)$ is an optimal control  and $\bar
x(\cdot)$ is the corresponding  optimal state.
\end{cor}

\begin{proof}
  Let $(u (\cdot), x (\cdot))$ be an arbitrary admissible pair.  From the
  condition 2, we  see that
  \eqref{eq:5.4} holds.
  Moreover, following  the same
  argument as   the proof of Theorem \ref{thm:4.3},
  \eqref{eq:4.9} and \eqref{eq:3.80} also
  hold.  Therefore, Putting \eqref{eq:3.80} and  \eqref{eq:5.4} into \eqref{eq:4.9}, we get
\begin{eqnarray}
J (u (\cdot)) - J (\bar u (\cdot)) \geq 0, \
\end{eqnarray}
which implies that  $\bar u (\cdot)$ is an optimal control  and $\bar
x(\cdot)$ is the corresponding optimal state. The proof is complete.

\end{proof}

\subsection{ Necessary Conditions of Optimality}

In this section we are going to represent
the necessary Pontryagin maximum principle of Problem  \ref{pro:4.1}.   To this end,
we need the following  variation formula.
\begin{thm}\label{them:4.5}
 Let  Assumption \ref{ass:1.1} and \ref{ass:1.2} be satisfied. Let $(\bar u (\cdot),
\bar x (\cdot))$  be an admissible pair associated
with the adjoint process $(\bar p(\cdot), \bar q(\cdot), \bar{\tilde q}(\cdot)).$
Then
\begin{eqnarray}\label{eq:4.15}
&& \frac{d}{d\epsilon} J ( \bar u (\cdot) + \epsilon ( u (\cdot) - \bar u (\cdot) ) ) |_{\epsilon=0} \nonumber \\
&& := \lim_{\epsilon \rightarrow 0^+}
\frac{ J ( \bar u (\cdot) + \epsilon ( u (\cdot) - \bar u (\cdot) ) ) - J( \bar u (\cdot) ) }{\epsilon} \nonumber \\
&& = {\mathbb E} \bigg [ \int_0^T \left < \mathbb E\big[{
\bar H}_u(t)+\mathbb E[\bar H_v(t)], u (t) - \bar u
(t) \right > d t \bigg ] \ .
\end{eqnarray}
where $\eps\in (0,1)$ and $u(\cdot)$
is any given admissible control.
\end{thm}
\begin{proof}
  Following an
  argument  similar to the proof of Theorem \ref{them:3.7},
  \eqref{eq:4.15} can be obtained by
  Lemma \ref{lem:4.2}. Here we do not
repeat it.  The proof is complete.
\end{proof}

 Then by Theorem \ref{them:4.5}, we get the following the necessary Pontryagin maximum principle of Problem  \ref{pro:4.1}.
\begin{thm} \label{thm:4.6}
Let Assumption \ref{ass:1.1} and \ref{ass:1.2} be satisfied.  Let $(\bar{u}(\cdot ); \bar{x}(\cdot ))$ be an optimal pair of Problem \ref{pro:4.1}
associated
with the adjoint process $(\bar p(\cdot), \bar q(\cdot), \bar{\tilde q}(\cdot)).$
.  Then
the optimality condition
\begin{align} \label{eq:4.16}
\Big \langle \mathbb E\big[{
\bar H}_u(t)+\mathbb E[\bar H_v(t)]|\mathscr F_t^Y\big], u-\bar u(t)\Big\rangle  \geq 0
\end{align}%
holds for all $u\in U$  and a.e. $(t,\omega )\in \lbrack 0,T]\times \Omega .$
\end{thm}

\begin{proof}
Since all admissible controls
are $\{\mathscr F^Y_t\}_{t\in \cal T}$-adapted
processes, from the property of conditional
expectation, Theorem \ref{them:4.5} and the optimality of $\bar u(\cdot)$,  we
deduce that
\begin{eqnarray*}
&& {\mathbb E} \bigg [ \int_0^T
 \langle {\mathbb E} [{ \bar  H}_u (t)+\mathbb E[\bar H_v(t)] | {\mathscr F}_t ^Y] ,
u (t) - \bar u (t) \rangle  d t \bigg ] \\
&& = {\mathbb E} \bigg[\int_0^T \langle { \bar  H}_u (t)+\mathbb E[\bar H_v(t)]
, u (t) - \bar u (t) \rangle \d t \bigg ] \\
&& = \lim_{\epsilon \rightarrow 0} \frac{J( \bar u (\cdot) + \epsilon (
u (\cdot) - \bar u (\cdot) ) ) - J ( \bar u (\cdot) )}{\epsilon} \geq 0, \
\end{eqnarray*}
which imply that \eqref{eq:4.16}
holds. The proof is complete.
\end{proof}

\section{Application}

In this section, we apply
our  stochastic maximum  principle to
solve  a  partial observed  stochastic  linear quadratic (LQ)
optimal control problem. Let us make it more precise below.
In this case, we assume the  state system is the
following linear mean-field  SDE
\begin{equation}\label{eq:5.1}
\left\{\begin {array}{ll}
  dX(t)=&(A_1(t)X(t)+A_2(t)\mathbb E [X(t)]
  +B_1(t)u(t)+B_2(t)\mathbb E [u(t)])dt
  \\&+(C_1(t)X(t)
  + C_2(t)\mathbb E [X(t)]
  +D_1(t)u(t)+D_2(t)\mathbb E [u(t)])dW(t)
  \\&+(F_1(t)X(t)
  +F_2(t)\mathbb E [X(t)]
  +G_1(t)u(t)+G_2(t)\mathbb E [u(t)])dW^u(t),
   \\x(0)=&x \in \mathbb R^n,
\end {array}
\right.
\end{equation}
with an observation

\begin{equation}\label{eq:5.2}
\displaystyle\left\{
\begin{array}{lll}
dY(t) & = & h(t)dt+dW^u(t),  \\_{}
\displaystyle Y(0) & = & 0,%
\end{array}
\right.
\end{equation}
and the cost functional  has the following
quadratic form:

 \begin{eqnarray}\label{eq:5.3}
\begin{split}
J ( u (\cdot) ) = &{\mathbb E} [
 \langle M_1 X (T), X(T) \rangle ]+{\mathbb E} [ \langle  M_2  \mathbb E[X (T)],
 \mathbb E[X (T)] \rangle ]
\\&+{\mathbb E} \bigg [ \int_0^T \langle Q_1 (s) X (s), X
(s) ) d s \bigg ]+ {\mathbb E} \bigg [ \int_0^T \langle Q_2 (s)  \mathbb E[X (s)],\mathbb E [X(s)] \rangle d s \bigg ]
\\&+ {\mathbb E} \bigg [ \int_0^T
 \langle N_1 (s) u (s), u (s) \rangle d s\bigg]+ {\mathbb E} \bigg [ \int_0^T \langle N_2 (s)
\mathbb E[u (s)], \mathbb E[u (s)] \rangle d s
\bigg ].
\end{split}
\end{eqnarray}
 In this case, our  control process  $u(\cdot)$  is said to be  an admissible  stochastic
process if  $u(\cdot)\in
M_{{\mathscr F}^{Y}}^2(0,T; \mathbb R^k).$  The set of all admissible controls
is also denoted by $U_{ad}^S.$ Note that
there is no constraint on our control process,
since it takes value in  $\mathbb R^k.$  Now
we make the basic assumptions on
the coefficients.

\begin{ass}\label{ass:5.1}
 The matrix-valued functions $A_1, A_2, C_1,  C_2, F_1,  F_2,
 Q_1,  Q_2:[0, T]\rightarrow \mathbb R^{n\times n};
B_1, \\B_2,D_1, D_2,  G_1,  G_2,:[0, T]\rightarrow \mathbb R^{n\times k}; N_1, N_2:[0, T]\rightarrow \mathbb R^{k\times k}; h:[0, T]\rightarrow \mathbb R$ are uniformly bounded measurable functions.
$M_1$ and $M_2$ are matrices in $\mathbb R^{n\times n}.$
\end{ass}

\begin{ass}\label{ass:5.2}
 The matrix-valued functions $Q_1, Q_1+ Q_2, N_1, N_1+N_2$ are a.e. nonnegative
matrices, and $M_1, M_1+M_2$ are nonnegative matrices.  Moreover, $N_1,N_1+N_2 $  uniformly positive, i.e. for $\forall u\in \mathbb R^m$ and a.s. $t\in [0, T]$,
$ \langle N_1(t)u, u \rangle \geq \delta \langle u, u\rangle
$ and $ \langle (N_1(t)+ N_2(t))u, u \rangle \geq \delta \langle  u, u\rangle,
$  for some positive constant
$\delta$.
\end{ass}

 Then our partial observed  mean-field
 LQ problem can be stated as follows.

 \begin{pro}\label{pro:5.1}

   Find an admissible control $\bar{u}(\cdot)$ such that
\begin{equation}  \label{eq:b7}
J(\bar{u}(\cdot))=\displaystyle\inf_{u(\cdot)\in U_{ad}^S}J(u(\cdot)),
\end{equation}
   subject to \eqref{eq:5.1},
   \eqref{eq:5.2} and \eqref{eq:5.3}.
 \end{pro}

It is easy to check that  under
Assumptions \ref{ass:5.1} and \ref{ass:5.2},
if we set

\begin{eqnarray}\label{eq:5.5}
  \begin{split}
    b(t,x,y,u,v)=&A_1(t)x+A_2(t)y
  +B_1(t)u+B_2(t)v,
   \\ g(t,x,y,u,v)=&C_1(t)x+C_2(t)y
  +D_1(t)u+D_2(t)v,
  \\\tilde g(t,x,y,u,v)=&
  F_1(t)x+F_2(t)y
  +G_1(t)u+G_2(t)v,\\
  h(t,x,y,u,v)=&h(t),
  m(x,y)=(M_1x, x)+(M_2y, y)
  \\l(t,x,y,u,v)=&(Q_1x, x)+(Q_2y, y)+
  (N_1u, u)+(N_2v, v).
  \end{split}
\end{eqnarray}
Problem \ref{pro:5.1} can be regarded as  a special
case of Problem \ref{pro:4.1}
 and  Assumptions \ref{ass:1.1} and \ref{ass:1.2}  for \eqref{eq:5.5} hold. Thus  Theorem \ref{thm:4.3}
and \ref{thm:4.6} can be applied to solve  Problem \ref{pro:5.1}.  In this case,  the Hamiltonian becomes
\begin{eqnarray} \label{eq:5.6}
\begin{split}
 & H(t,x,y,u,v,p,q,\tilde q)
  \\= &
  \langle  p, A_1(t)x+A_2(t)y
  +B_1(t)u+B_2(t)v- h(t)(F_1(t)x+F_2(t)y
  +G_1(t)u+G_2(t)v ) \rangle
  \\&
    +\langle  q, C_1(t)x+C_2(t)y
  +D_1(t)u+D_2(t)v \rangle
  +\langle  \tilde q, F_1(t)x+F_2(t)y
  +G_1(t)u+G_2(t)v \rangle
  \\& +\langle Q_1x, x\rangle +\langle Q_2y, y
  \rangle +
  \langle N_1u, u\rangle +\langle N_2v, v\rangle.
  \end{split}
\end{eqnarray}
For any admissible pair $(u(\cdot), x(\cdot)),$
the corresponding adjoint equation becomes

 \begin {equation}\label{eq:5.7}
\left\{\begin{array}{lll}
dp(t)&=&-\bigg[(A^\top_1(t)-h(t)F_1^\top(t))p(t)+ (A_2^\top(t)-h(t)F_2^\top(t))\mathbb E [p(t)]+C^{\top}_1(t)q(t)
\\&&+ C^{\top}_2(t)\mathbb E[q(t)]
+ F^{\top}_1(t)\tilde q(t)+F^{\top}_2(t)\mathbb E[\tilde q(t)]+2Q_1(t)X(t)
+2Q_2(t)\mathbb E[X(t)]\bigg]dt
\\&&+q(t)dW(t)+\tilde q(t)dY(t),
 \\p(T)&=&2M_1X(T)+2M_2\mathbb E[X(T)].
\end{array}
  \right.
  \end {equation}
  The following result gives the existence and uniqueness of the optimal control of Problem \ref{pro:5.1}.

\begin{thm}
  Let Assumptions \ref{ass:5.1} and
  \ref{ass:5.2} be satisfied. Then
  Problem \ref{pro:5.1} has a unique
  optimal control.
\end{thm}
\begin{proof}

Since  the admissible control set $U_{ad}^S=
M_{{\mathscr F}_{Y}}^2(0,T; \mathbb R^k)$  is
a Hilbert space, thus a reflexive Banach  space,
to prove the existence and uniqueness of the optimal control,  by the classic optimality principle
 (see Proposition 2.12 of Ekeland and T\'emam (1976)), it needs only to prove
 that over $U_{ad}^S,$
the cost functional $J(u(\cdot))$
is the strictly convex, coercive and  lower-semi continuous. Indeed, by the a priori
estimate \eqref{eq:1.8} and \eqref{eq:1.15}, over
$U_{ad}^S,$ we can show that the cost functional
$J ( u (\cdot) )$ is continuous and hence lower-semi continuous. On the other hand, since the weighting matrices in the cost
functional are not random,  from the
definition of  $J ( u (\cdot) )$ (see \eqref{eq:5.3}) and by a simple calculation, we can get that
  \begin{eqnarray}\label{eq:2.4}
\begin{split}
 J( u(\cdot))=& \displaystyle \mathbb E\bigg[\int_0^T\bigg(\langle Q_1(t)(X(t)-\mathbb E[X(t)]),
X(t)-\mathbb E[X(t)])\rangle + \langle  (Q_1+{
Q}_2)(t)\mathbb E[X(t)], \mathbb E[X(t)]\rangle
 \\&+\langle N_1(t)(u(t)-\mathbb E[u(t)]), u(t)-\mathbb E [u(t)]\rangle
 +\langle  (N_1(t)+
{N}_2(t))\mathbb E[u(t)], \mathbb E[u(t)]\rangle \bigg )dt\bigg]
\\&+\mathbb E\bigg[\langle
M_1(X(T)-\mathbb E [X(T)], X(T)-\mathbb E [X(T)]\rangle +\langle  (M_1+{ M}_2)\mathbb E[X(T)], \mathbb
E[X(T)]\rangle \bigg].
\end{split}
\end{eqnarray}
Thus  the cost functional $J(u(\cdot))$ over
 $U_{ad}^S$ is convex from the nonnegativity of the
 $N_1, N_1+N_2, Q_1, Q_1+ Q_2, M_1, M_1+ M_2 $. Actually, since
 $N_1 $ and  $N_1+N_2$ are uniformly positive, $J(u(\cdot))$ is strictly
convex. Furthermore, it follows from the
nonnegativity of $M_1, M_1+M_2$ and $Q_1, Q_1+Q_2$ and the
uniformly strictly positivity  of $N_1, N_1+N_2$, that
\begin{eqnarray}\label{eq:2.5}
  \begin{split}
    J(u(\cdot)) \geq& \mathbb E\bigg[\int_0^T
    \bigg (\langle N(t)(u(t)-\mathbb E[u(t)]), u(t)-\mathbb E[u(t)]\rangle
 +\langle  (N(t)+\bar
{N}(t))\mathbb E[u(t)], \mathbb E[u(t)]\rangle \bigg )dt\bigg]
\\ \geq &\delta \mathbb E\bigg[\int_0^T\langle u(t)-\mathbb E[u(t)], u(t)-\mathbb E[u(t)]\rangle dt\bigg]+
 \delta\mathbb E\bigg[\int_0^T \langle \mathbb E[u(t)], \mathbb E[u(t)]\rangle dt\bigg]
 \\=& \delta\mathbb E \bigg[\int_0^T |u(t)|^2dt\bigg]
\\=&\delta ||u(\cdot)||^2_{U_{ad}^S},
  \end{split}
\end{eqnarray}
which implies that $J ( u (\cdot) )$ is coercive, i.e.,
\begin{eqnarray*}
\lim_{ \| u (\cdot) \|_{U_{ad}^S} {\rightarrow \infty} } J ( u (\cdot) ) = \infty .
\end{eqnarray*}
In summary, the cost functional $J(u(\cdot))$ is strictly convex, coercive, lower-semi continuous  over the reflexive Banach space  $U_{ad}^S.$
The proof is complete.
\end{proof}

In the following,  applying the maximum
principle to our LQ problem, we give the dual presentation of the optimal control in terms
of the corresponding adjoint process.

\begin{thm}\label{thm:5.3}
Let Assumptions \ref{ass:5.1} and
  \ref{ass:5.2} be satisfied.
 Then, a necessary and
sufficient condition for an admissible pair $(u(\cdot); x(\cdot))$ to be an optimal pair of  Problem \ref{pro:5.1}  is that the admissible
control $u(\cdot)$ satisfies
\begin{eqnarray} \label{eq:5.9}
  \begin{split}
 &2N_1(t)u(t)+2N_2(t)\mathbb E[u(t)]+
 (B^\top_1(t)-h(t)G_1^\top(t))
 \mathbb E[p(t)|\mathscr F^{Y}_{t}]+(B^\top_2(t)
 -h(t)G^\top_2(t))\mathbb E [p(t)]
 \\&~~~~~~+ D^{\top}_1(t)\mathbb E[q(t)|\mathscr F^{Y}_{t}]
 +D^{\top}_2(t) \mathbb E [q(t)]=0, \quad a.e. a.s.,
     \end{split}
\end{eqnarray}
where $(p(\cdot),q(\cdot), \tilde q(\cdot)) $ is the
solution to the adjoint equation \eqref{eq:5.7}
corresponding to $(u(\cdot), X(\cdot))$.
\end{thm}

\begin{proof}
 For the necessary part, let $(u(\cdot), x(\cdot))$
 be an optimal pair associated with
 the adjoint process $(p(\cdot),q(\cdot), \tilde q(\cdot)). $   Since
 there is no constraints on the control processes, then  from the
 necessary optimality condition \eqref{eq:4.16}
 (see Theorem \ref{thm:4.6}), we get  that
 \begin{eqnarray} \label{eq:5.10}
 \begin{split}
 &\mathbb E\bigg[H_u(t,x(t),\mathbb E[x(t)], u(t), \mathbb E[u(t)], p(t), q(t), \tilde q(t))|\mathscr
F_t^{Y}\bigg]
\\&+\mathbb E\bigg[H_v(t,x(t),\mathbb E[x(t)], u(t), \mathbb E[u(t)], p(t), q(t), \tilde q(t))\bigg]=0,
\end{split}
 \end{eqnarray}
 which leads to
 \eqref{eq:5.9} ( recalling  the definition  \eqref{eq:5.6}
  of Hamiltonian  $H$).

  For the sufficient part,  let $(u(\cdot), X(\cdot))$
 be an admissible pair associated with
 the adjoint process $(p(\cdot),q(\cdot), \tilde q(\cdot)) $ and assume  the condition
 \eqref{eq:5.9} holds.   From the
 definition of $H$ (see \eqref{eq:5.6}),
   the condition
 \eqref{eq:5.9} implies \eqref{eq:5.10} holds.
  Thus,  since   any
  admissible control is
  $ \mathscr F_{t}^Y$-adapted process.
   by\eqref{eq:5.10}
  , for any other
  admissible control $v(\cdot),$
  from the property of conditional
  expectation, we have
  \begin{eqnarray} \label{eq:4.13}
  \begin{split}
  &\mathbb E\bigg[\bigg\langle  v (t) - u (t), H_u(t,x(t),\mathbb E[x(t)], u(t), \mathbb E[u(t)], p(t), q(t), \tilde q(t))
\\&~~~~~~+\mathbb E\bigg[H_v(t,x(t),\mathbb E[x(t)], u(t), \mathbb E[u(t)], p(t), q(t), \tilde q(t))\bigg]\bigg\rangle\bigg]
\\=&\mathbb E\bigg[\bigg\langle  v (t) - \bar u (t), \mathbb E\bigg[H_u(t,x(t),\mathbb E[x(t)], u(t), \mathbb E[u(t)], p(t), q(t), \tilde q(t))|\mathscr
F_t^{Y}\bigg]
\\&~~~~~~+\mathbb E\bigg[H_v(t,x(t),\mathbb E[x(t)], u(t), \mathbb E[u(t)], p(t), q(t), \tilde q(t))\bigg]\bigg\rangle\bigg]
\\&=0,
\end{split}
\end{eqnarray}
which implies that
the condition 3 in Theorem \ref{thm:4.3}
holds. Moreover, under Assumptions \ref{ass:5.1}
and \ref{ass:5.2}, it is easy to check that all other conditions in  Theorem \ref{thm:4.3} are satisfied.
Therefore, by  Theorem \ref{thm:4.3}, we conclude that $(u(\cdot), x(\cdot))$ is an optimal control pair. The proof is complete.
\end{proof}
From the above, we end up the following
optimality system
\begin{equation}\label{eq:5.12}
\left\{\begin {array}{ll}
  dX(t)=&(A_1(t)X(t)+A_2(t)\mathbb E [X(t)]
  +B_1(t)u(t)+B_2(t)\mathbb E [u(t)])dt
  \\&+(C_1(t)X(t)
  +\bar C_2(t)\mathbb E [X(t)]
  +D_1(t)u(t)+D_2(t)\mathbb E [u(t)])dW(t)
  \\&+(F_1(t)X(t)
  +F_2(t)\mathbb E [X(t)]
  +G_1(t)u(t)+G_2(t)\mathbb E [u(t)])dW^u(t),\\
  dY(t)  =& h(t)dt+dW^u(t), \\
  dp(t)=&-\bigg[(A^\top_1(t)-h(t)F_1^\top(t))p(t)+ (A_2^\top(t)-h(t)F_2^\top(t))\mathbb E [p(t)]+C^{\top}_1(t)q(t)
\\&+ C^{\top}_2(t)\mathbb E[q(t)]
+ F^{\top}_1(t)\tilde q(t)+F^{\top}_2(t)\mathbb E[\tilde q(t)]+2Q_1(t)X(t)
+2Q_2(t)\mathbb E[X(t)]\bigg]dt
\\&+q(t)dW(t)+\tilde q(t)dY(t),
 \\x(0)=&x ,p(T)=2M_1X(T)+2M_2\mathbb E[X(T)], Y(0)=0,\\
 2N_1(t)u(t)&+2N_2(t)\mathbb E[u(t)]+
 (B^\top_1(t)-h(t)G_1(t))
 \mathbb E[p(t)|\mathscr F^{Y}_{t}]+(B^\top_2(t)
 -h(t)G^\top_2(t))\mathbb E [p(t)]
 \\&+ D^{\top}_1(t)\mathbb E[q(t)|\mathscr F^{Y}_{t}]
 +D^{\top}_2(t) \mathbb E [q(t)]=0.
\end {array}
\right.
\end{equation}

This is a fully coupled  forward-backward stochastic
differential equations of mean-field type.
 Note that the coupling comes from the last relation (which is essentially the maximum condition in the Pontryagin type maximum principle). The 5-tuple  $(
u(\cdot), x(\cdot), p(\cdot),q(\cdot), \tilde q(\cdot)) $ of $\mathscr F_{t}$-adapted processes
satisfying the above is called an adapted solution of \eqref{eq:5.12}.   Then by Theorem \ref{thm:5.3},
we can  directly  obtain the following   equivalence
 between the solvability of optimality system \eqref{eq:5.12} and the existence and unique of
 the optimal
 control of Problem \ref{pro:5.1}.
\begin{cor}
  Let Assumptions \ref{ass:5.1} and
  \ref{ass:5.2} be satisfied.
 Then, a necessary and
sufficient condition for  that
the optimality system \eqref{eq:5.12}
has a unique solution strong solution  $(u(\cdot), x(\cdot),p(\cdot),
 q(\cdot),\tilde q(\cdot))\\\in M_{\mathscr{F^Y}}^2(0,
T;\mathbb R^k)\times S_{\mathscr{F}}^2(0,
T;\mathbb R^n)\times S_{\mathscr{F}}^2(0, T;\mathbb R^n)\times M_{\mathscr{F}}^2(0,T;\mathbb R^{n})\times M_{\mathscr{F}}^2(0,T;\mathbb R^{n})$  is that
 $(u(\cdot); x(\cdot))$ is  a unique optimal pair of  Problem \ref{pro:5.1}.
\end{cor}

\begin{rmk}
In summary, the optimality system \eqref{eq:5.12}
completely characterizes the optimal control of Problem \ref{pro:5.1}. Therefore, solving Problem \ref{pro:5.1} is equivalent to solving
the optimality system, moreover, the unique optimal control
can be given by \eqref{eq:5.9}.
Taking expectation on \eqref{eq:5.9},  we have
\begin{eqnarray} \label{eq:3.8}
  \begin{split}
 &2(N_1(t)+N_2(t))\mathbb E[u(t)]
 +(B^\top_1 (t)+ B^\top_2 (t)-h(t)G^\top_1 (t)
 -h(t)G^\top_2 (t)) \mathbb E [p(t)]
 \\&\quad\quad\quad+ (D^{\top}_1(t)
 +D^{\top}_2(t)) \mathbb E [q(t)]=0, \quad a.e. a.s.,
     \end{split}
\end{eqnarray}
which implies

\begin{eqnarray} \label{eq:3.9}
  \begin{split}
 &\mathbb E[u(t)]
 =-\frac{1}{2}(N_1(t)+ N_2(t))^{-1}\bigg[(B_1 (t)+ B_2 (t)-h(t)G_1 (t)
 -h(t)G_2 (t))^\top \mathbb E [p(t)]
 \\&\quad \quad\quad\quad\quad\quad\quad
 \quad+ (D_1(t)
 +D_2(t))^\top \mathbb E [q(t)]\bigg],  a.s.
     \end{split}
\end{eqnarray}
Putting \eqref{eq:3.9} into \eqref{eq:5.9},
we get that
\begin{eqnarray} \label{eq:3.6}
  \begin{split}
 &2N_1(t)u(t)=-2N_2(t)\mathbb E[u(t)]-
 (B^\top_1(t)-h(t)G_1^\top(t))
 \mathbb E[p(t)|\mathscr F^{Y}_{t}]-(B^\top_2(t)
 -h(t)G^\top_2(t))\mathbb E [p(t)]
 \\&~~~~~~- D^{\top}_1(t)\mathbb E[q(t)|\mathscr F^{Y}_{t}]
 -D^{\top}_2(t) \mathbb E [q(t)],
     \end{split}
\end{eqnarray}
which imply that the
optimal control $u(\cdot)$
has the following explicit
dual presentation

\begin{eqnarray} \label{eq:4.17}
  \begin{split}
 \bar u(t)=&-\frac{1}{2}N^{-1}_1(t)\bigg\{
 (B^\top_1(t)-h(t)G_1^\top(t))
 \mathbb E[p(t)|\mathscr F^{Y}_{t}]+(B^\top_2(t)
 -h(t)G^\top_2(t))\mathbb E [p(t)]
 \\&~~~~~~+D^{\top}_1(t)\mathbb E[q(t)|\mathscr F^{Y}_{t}]
 +D^{\top}_2(t) \mathbb E [q(t)]
\\&+ N_2(t)(N_1(t)+ N_2(t))^{-1}\bigg[(B_1 (t)+ B_2 (t)-h(t)G_1 (t)
 -h(t)G_2 (t))^\top \mathbb E [p(t)]
 \\&\quad \quad\quad\quad\quad\quad\quad
 \quad+ (D_1(t)
 +D_2(t))^\top \mathbb E [q(t)]\bigg]\bigg\}
, \quad a.e. a.s.
     \end{split}
\end{eqnarray}

\end{rmk}

In the following, we will give
the state feedback representation
of the optimal control.

\begin{thm}
  Let Assumptions \ref{ass:5.1} and
  \ref{ass:5.2} be satisfied.
  Let $(\bar u(\cdot), \bar x(\cdot))$ be
the optimal pair.
  Then the optimal control $\bar u(\cdot)$ has the following
  state feedback representation:
  \begin{eqnarray}  \label{eq:5.17}
    \begin{split}
     & \bar u(t)
 \\=&-\Sigma_0^{-1}(t)\bigg[
 (B^\top_1(t)-h(t)G_1^\top(t)) P(t)
+ D^{\top}_1(t)P(t)C_1(t)\bigg]\bigg[\mathbb E[\bar x (t)|\mathscr F^Y_{t}]
-\mathbb E[\bar x(t)]\bigg]
 \\&-\Sigma_2^{-1}(t)\bigg[(B^\top_1(t)
    + B^\top_2(t)-h(t)(G_1^\top(t)+G_2^\top(t))) \Pi(t)
    \\&\quad\quad\quad\quad\quad+ (D^{\top}_1(t)
 +D^{\top}_2(t)) P(t)(C_1(t)+ C_2(t))\bigg]\mathbb E[\bar x(t)],
    \end{split}
  \end{eqnarray}
 where
  \begin{eqnarray}
  \begin{split}
  \Sigma_0(t)=&2N_1(t)+ D^{\top}_1(t)PD_1(t),
   \\  \Sigma_2(t)=&2(N_1(t)+N_2(t))+(D^{\top}_1(t)
 +D^{\top}_2(t)) P(t)(D_1(t)+ D_2(t)),
  \end{split}
\end{eqnarray}
  $P(\cdot)$ and $\Pi(\cdot)$ are
  the solutions to the following
  Riccati equations, respectively:
  \begin {equation}\label{eq:4.20}
\left\{\begin{array}{lll}
 &(\dot{P}(t)+P(t)A_1(t)+A_1^\top(t) P(t)
 +C^{\top}_1(t)P(t)C(t)+2Q_1(t)
\\&~~~-\bigg [P(t)(B_1(t)-h(t)G_1(t))
+ C^{\top}_1(t)P(t)D_1(t)\bigg]\Sigma_0^{-1}(t)
\\&~~~~~~~\cdot\bigg [(B^\top_1(t)-h(t)G_1^\top(t)) P(t)
+ D^{\top}_1P(t)C_1(t)\bigg]
=0,
\\& P(T)=M_1
 \end{array}
  \right.
  \end {equation}
  and \begin {equation}\label{eq:4.21}
\left\{\begin{array}{lll}
 &\dot{\Pi}(t)+\Pi(t)(A_1(t)+A_2(t))
  +(A^\top_1(t)+A^\top_2(t))\Pi(t)+(C^{\top}_1(t)+  C^{\top}_1(t))P(t)(C_1(t)+C_2(t))
  \\&+2(Q_1+Q_2)
  \\&-\bigg[\Pi(t)(B_1(t)+B_2(t)
  -h(t)(G_1(t)+G_2(t)))
 +(C^{\top}_1(t)
 + C^{\top}_1(t))P(t) (D_1(t)+D_2(t))\bigg]
 \\&\quad \cdot\Sigma_2^{-1}(t)
\cdot \bigg[(B^\top_1(t)
    +B^\top_2(t)-h(t)(G_1^\top(t)+G_2^\top(t))) \Pi(t)+ (D^{\top}_1(t)
 +D^{\top}_2(t)) P(C_1(t)+C_2(t))\bigg]
 \\&=0,
 \\&
 \Pi(T)=M_1+M_2.
\end{array}
  \right.
  \end {equation}
  Moreover,
 \begin{eqnarray} \label{eq:4.30}
   \begin{split}
     \displaystyle \inf_{u(\cdot)\in {U_{ad}^S}}J(u(\cdot))=\langle \Pi(0)x, x\rangle .
   \end{split}
 \end{eqnarray}
\end{thm}

\begin{proof}

 To unburden our notation,
 define
 \begin{eqnarray}
   \begin{split}
    \hat B_1= B_1(t)-h(t)G_1(t), \hat B_2= B_2(t)-h(t)G_2(t).
   \end{split}
 \end{eqnarray}

The proof can be obtained by
the classic  technique of completing squares.
  Indeed,
  let $(u(\cdot), x(\cdot))$ be
  any given admissible pair.
   From Yong (2013), we  know that
   the Riccati equations \eqref{eq:4.20}
   and \eqref{eq:4.21} have
   a unique solution $P(\cdot)$
   and $\Pi(\cdot),$ respectively.
   Then following the same argument as that
  of Theorem 4.2 of Yong (2013), we get that (suppressing $t$)

  \begin{eqnarray}
    \begin{split}
      &J(u(\cdot))-\langle \Pi(0)x, x\rangle
     \\ =&\mathbb E\int_0^T\bigg\{
     \bigg|\Sigma_0^{\frac{1}{2}}\bigg[u-\mathbb
     E[u]+\Sigma_0^{-1}\big(
     \hat B^\top_1 P
+ D^{\top}_1PC_1\big)\big(X
-\mathbb E[X]\big)\bigg]\bigg|^2\bigg\}dt
\\&+\mathbb E\int_0^T\bigg\{
     \bigg|\Sigma_2^{\frac{1}{2}}\bigg[\mathbb
     E[u]+\Sigma_2^{-1}\bigg((
     \hat B^\top_1+
     \hat B^\top_2) P
+ (D^{\top}_1+D^\top_2)P(C_1+C_2)\bigg)\mathbb E[X]\bigg]\bigg|^2\bigg\}dt.
    \end{split}
  \end{eqnarray}
 Thus by the well-known Kallianpur–Striebel formula  in Kallianpur (2013) , we know that
 the minimum  $J(u(\cdot))$ over all
 $\mathscr F^Y_t$-measurable process $u(t)$
 is  attained at
 \begin{equation}  \label{eq:4.23}
   u(t)-\mathbb E[u(t)]=-\Sigma_0^{-1}\big(
   \hat B^\top_1 P
+ D^{\top}_1PC_1\big)\big(
\mathbb E[X(t)|\mathscr F_t^Y]
-\mathbb E[X]\big)
 \end{equation}
 and
 \begin{equation}\label{eq:4.24}
 \mathbb
     E[u]=-\Sigma_2^{-1}(
     \hat B^\top_1+\hat B^\top_2) P
+ (D^{\top}_1+D^\top_2)P(C_1+C_2)\mathbb E[u(t)],
 \end{equation}
 and   the minimum value is $\langle \Pi(0)x, x\rangle.$
 Therefore, combining
 \eqref{eq:4.23} and
 \eqref{eq:4.24}, we get that
 the optimal control $\bar u(\cdot)$ has
 the  state feedback representation
 \eqref{eq:5.17}
 The proof is complete.

\end{proof}

\end{document}